\documentclass[12pt]{article}
\usepackage[T1]{fontenc}
\usepackage[latin9]{inputenc}
\usepackage{geometry}
\usepackage{color}
\definecolor{document_fontcolor}{rgb}{0.5625, 0, 0.92578125}
\usepackage{amsmath}
\usepackage{amssymb}
\usepackage{wasysym}
\usepackage{esint}

\usepackage{amsthm}
\usepackage{enumerate}

\newtheorem{Thm}{Theorem}[section]
\newtheorem{LEM}[Thm]{Lemma}
\newtheorem{Prop}[Thm]{Proposition}
\newtheorem{Rem}[Thm]{Remark}
\usepackage{graphicx}
\makeatletter
 \usepackage{xcolor}
\definecolor{mauve}{RGB}{0, 180, 0}
\usepackage{cases}
\makeatother

\usepackage[english]{babel}

\begin{document}

\title{Approximations for time-dependent distributions in Markovian fluid
models}

\author{Sarah Dendievel\thanks{Universit\'e Libre de Bruxelles, D\'epartement d'Informatique,
CP~212, Boulevard du Triomphe, 1050 Bruxelles, Belgium; Sarah.Dendievel@ulb.ac.be, latouche@ulb.ac.be.} \and Guy Latouche$^*$}

\date{June 2014}
\maketitle
\begin{abstract}
In this paper we study 
the distribution of the level  at time $\theta$
of Markovian fluid queues and Markovian continuous time random walks,
the  maximum  (and minimum) level over $[0,\theta]$,
and their joint distributions.
We approximate 
$\theta$
by a random variable $T$ with Erlang distribution
 and we use an alternative
way, with respect to the usual Laplace transform approach, to compute the
distributions.
 We present probabilistic interpretation of the equations
and  provide a numerical illustration.
 \end{abstract}
\begin{description}
\item [{Keywords:}] 
Markov modulated fluid,
Erlangization approximations,
Distribution in finite time,
Joint distribution
\end{description}

\section{Introduction}

In the literature, Markovian fluid models have been analyzed for many years.
One of the first papers appeared in the sixties, for instance with 
Loynes
\cite{loynes1962continuous} studying the continuous-time behavior of queues.
In the eighties, Markovian fluid models started to be more extensively studied, in particular much work has been dedicated to the study of the stationary distribution, see for instance 
Rogers
\cite{rogers1994fluid}
and
Asmussen
 \cite{asmussen1995stationary}.
In the present paper, we explore a method to compute time-dependent distributions.

Intuitively, a Markovian fluid model
 represents the evolution
in time of some liquid level 
in a buffer: taps allow the
liquid to flow in and out at different rates.
The buffer may have a finite or an infinite capacity.
The flow is controlled by an underlying Markov chain
$\{\varphi(t):t\in\mathbb{R}^{+}\}$
with a finite state space $\mathcal{S}$,
called the \textit{phase process} :
when the Markov chain is in phase $i\in\mathcal{S}$, the level of the buffer increases with a constant rate $c_{i}$,
if $c_{i}$ is strictly positive,
or it decreases with the rate $c_{i}$,
if $c_{i}$ is strictly negative.
The \textit{level} $X(t)$ at time $t$
may be expressed as follows
\begin{equation}
X(t)=X(0)+\int_{0}^{t}c_{\varphi(s)}\mathrm{d}s.
\label{XEvolution}
\end{equation}
The process $\{\left(X(t),\varphi(t)\right):t\in\mathbb{R}^{+}\}$ is called the \textit{Markovian continuous time random walk} in this paper : it is an unrestricted
process and the level may become negative as well as positive.
The \textit{Markovian fluid queue} denoted by $\{(Z(t),\varphi\left(t\right)):t\in\mathbb{R}^{+}\}$
is related to the random walk in the following way : during
the intervals of time when  $Z\left(t\right)$ $=0$ and the rate at
time $t$ is negative, the level remains equal to zero.
The level $Z(t)$$ $ at time $t$ can be expressed as follows 
\begin{equation}
Z(t)=X(t)-\inf_{0\leq v\leq t}X(v).
\label{ZEvolution}
\end{equation}
 As the environments
remain Markovian in the whole paper, we refer
to process $\{(X(t),\varphi(t)):t\in\mathbb{R}^{+}\}$
by the name \textit{random
walk}
and to
the process $\{(Z(t),\varphi(t)):t\in\mathbb{R}^{+}\}$
by the name \textit{fluid queue}.

In this work, we focus on determining the distribution of 
$X(\theta)$, and $Z(\theta)$, at a finite time $\theta>0$.
In the literature, such time-dependent distributions  have been studied using Laplace transforms :
Ahn 
and
Ramaswami
\cite{ahn2006transient} derive time-dependent
distributions of a fluid queue  in terms of the
transform matrix of the busy period duration,
i.e. the matrix
$\Psi(s)$ is such that
$\Psi_{ij}(s)=\mathbb{E}[\exp(-s\tau)1_{\{\varphi(\tau)=j\}}|X(0)=0,\varphi(0)=i]$
for
$\tau = \inf\{t>0 : Z(t) = 0\}$,
 $i\in \mathcal{S}_{+}$ and $j\in \mathcal{S}_{-}.$
Here, we use arguments based on the Erlangization method 
and so avoid Laplace transform calculations.
The idea is to replace the fixed time $\theta$
by an Erlang-distributed random variable $T$
such that $\mathbb{E}[T]=\theta$.
Its advantage is that in
so doing, we replace integral equations by
linear equations.
The Erlangization method has been suggested by
Asmussen et al.
 \cite{asmussen2002erlangian}
 for ruin problems,
 Ramaswami et al.
 \cite{ramaswami2008erlangization} determine the return probability to the initial level before the end of the Erlang period
 and 
 Stanford \textit{et al.}
 \cite{stanford2011erlangian}
 analyze the distribution of the time to ruin.
 
 In addition to the distribution of $X(T)$,
 we compute the
 joint distribution of 
 $X(T)$
 {\it{and}}
 the minimum level during $[0,T]$.
We do the same analysis for the joint distribution of 
$X(T)$ and the  maximum level during $[0,T]$.
The Erlangization technique is of particular importance in this case of joint distributions
 because it replaces
integral equations,
that we would have to manipulate if  the time $\theta$ would be deterministic,
 by 
linear equations.
As a consequence,
the resulting probability distributions are more easily computable in the randomized version of time.

We observe that $X(T)$ has a bilateral phase-type (BPH) distribution, the density function of the BPH  has been determined in Ahn and Ramaswami \cite{ahn2005bilateral}.
Here, we go beyond the distribution of the level at time $T$ as mentionned before and we follow a different approach. We shall discuss about the differences and the similarities between the BPH and the distribution at maturity of the random walk in more details at the end of Section \ref{TDDRW}.

In the next section, we define precisely the random
walk and the fluid queue.
In Section  \ref{SECERL}, we explain the Erlangization
method and show how to combine this method with the Markovian fluid
models.
The equations for 
$X(T)$
and
$Z(T)$
are different because of
the constraint at level $0$
in the second process,
and we analyze their time-dependent behavior in two separate sections :  Section \ref{TDDRW}
and   \ref{TDDFQ}
respectively.
We conclude with an illustrative example in Section \ref{ILL}.

\section{Fluid models}
Consider the random walk $\{(X(t),\varphi\left(t\right)):t\in\mathbb{R}^{+}\}$, where
$X(t)$ is defined in
 	(\ref{XEvolution}),
and the fluid queue
	 $\{(Z(t),\varphi(t)):t\in\mathbb{R}^{+}\}$,
 where
$Z(t)$ is defined in 
	 (\ref{ZEvolution}).
For both of the fluid  models, we assume that the input rate $c_i$ is
different from zero for all $i\in \mathcal{S}=\{1,...,m\}$. We partition $\mathcal{S}$ into $\mathcal{S}_{+}\cup \mathcal{S}_{-}$
with $\mathcal{S}_{+}=\{i\in \mathcal{S}:c_{i}>0\}$ and $\mathcal{S}_{-}=\{i\in \mathcal{S}:c_{i}<0\}$.
Similarly, we define
the fluid rate matrix
$C=\mbox{diag}(c_{1},...,c_{m})$
and partition $C$ into 
 $C_{+}$ and $C_{-}$. The infinitesimal generator of $\{\varphi\left(t\right):t\in\mathbb{R}^{+}\}$
is denoted by $A$ and is written,
possibly after permutation of rows and
columns,  as 
\begin{equation}
A=\left[\begin{array}{cc}
A_{++} & A_{+-}\\
A_{-+} & A_{--}
\end{array}\right].
\label{GeneratorA}
\end{equation}
We define the joint distribution functions
$F_{j}(x,t)=\mathbb{P}\left[X\left(t\right)\leq x,\varphi\left(t\right)=j\right]$
for the random walk and 
the joint distribution functions
$H_{j}(x,t)=\mathbb{P}\left[Z\left(t\right)\leq x,\varphi\left(t\right)=j\right]$
for the fluid queue.
The next theorem gives a differential equation for
$F(\cdot,\cdot)$.
This result is well-known 
and a proof may by find in 
Mitra
\cite{mitra1988stochastic}.

\begin{Thm}
For all
$x \in \mathbb{R}$,
 $j\in  \mathcal{S}$, the joint distribution functions
$F_{j}(x,t)$ are a solution of the system
of partial differential equations
\begin{equation*}
\frac{\partial}{\partial t}F_{j}(x,t)=\sum_{i\in \mathcal{S}}F_{i}(x,t)A_{ij}-c_{j}\frac{\partial}{\partial x}F_{j}(x,t).
\end{equation*}
The joint distribution functions $H_{j}(x,t)$ 
are a solution of the same system for $x\geq0$ and $j\in  \mathcal{S}$.
\hfill $\square$
\end{Thm}

Two pairs of matrices play an important role in the next sections. These are matrices of  first return probabilities to the
initial level and the infinitesimal generators of monotone records.
Denote by $\tau_{+}(x)$ and $\tau_{-}(x)$ the two
first passage times 
\[
\tau_{+}(x)=\inf\{t>0:X(t)>x\}
\qquad
\mbox{ and }
\qquad 
\tau_{-}(x)=\inf\{t>0:X(t)<x\}.
\]
We denote by $\Psi_{ij}$ the probability that, starting from $\left(x,i\right)$
at time $0$, with $x\in\mathbb{R}$ and $i\in \mathcal{S}_{+}$, the random
walk returns to  level $x$ in a finite time and does so in phase
$j$, with $j\in \mathcal{S}_{-}$ : 
\begin{equation}\label{DefdePsi}
\Psi_{ij}=\mathbb{P}\left[\tau_{-}(x)<\infty,\varphi\left(\tau_{-}(x)\right)=j|X\left(0\right)=x,\varphi\left(0\right)=i\right].
\end{equation}
As we assume that the
process has spatial homogeneity, the matrix  $\Psi$ of
 \textit{first return probability from above}  does not depend
on the level $x$.
Similarly, the matrix $\hat{\Psi}$ of \textit{first return probabilities from below}
has the components 
\begin{equation}\label{DefdePsiHat}
\hat{\Psi}_{ij}=\mathbb{P}\left[\tau_{+}(x)<\infty,\varphi\left(\tau_{+}(x)\right)=j|X\left(0\right)=x,\varphi\left(0\right)=i\right]
\end{equation}
where $i\in \mathcal{S}_{-}$ , $j\in \mathcal{S}_{+}$.
Note that if $X(0)=x$ and if the initial phase belongs to $\mathcal{S}_{+}$,
then $\tau_{+}(x)=0$;
if, on the contrary, the initial phase belongs to $\mathcal{S}_{-}$,
then $\tau_{-}(x)=0.$
The following theorem is equivalent to  Theorem 1 in
Rogers
 \cite{rogers1994fluid}.

\begin{Thm}
The matrix $\Psi$ is the minimal nonnegative solution
of the Riccati equation
\begin{equation*}
C_{+}^{-1}A_{+-}+C_{+}^{-1}A_{++}\Psi+\Psi\left\vert C_{-}\right\vert ^{-1}A_{--}+\Psi\left\vert C_{-}\right\vert ^{-1}A_{-+}\Psi=0.
\end{equation*}
The matrix  $\hat{\Psi}$ is the minimal nonnegative solution
of the Riccati equation
	\begin{equation*}
	\left\vert C_{-}\right\vert ^{-1}A_{-+}
	+\left\vert C_{-}\right\vert ^{-1}A_{--}\hat{\Psi}
	+\hat{\Psi} C_{+}^{-1}A_{++}
	+\hat{\Psi} C_{+}^{-1}A_{+-}\hat{\Psi}
	=0.
	\end{equation*}
	\hfill $\square$
\end{Thm}

Define $\vartheta(x)=\tau_{-}(-x)$, $R(x)=\varphi(\vartheta(x))$
for $x\geq0$.
The process $\{R(x):x\geq0\}$  is Markovian, it corresponds to the phase process observed only during those
intervals of time in which $X(t)=\min_{0\leq s\leq t}X(s)$, and is called
the process of \textit{downward records}.
Its generator, the matrix $U$,
may be expressed in terms of the matrix of first
return probability from above as 
\begin{equation*}
U=\left|C_{-}^{-1}\right|A_{--}+\left|C_{-}^{-1}\right|A_{-+}\Psi,
\end{equation*}
see for instance
da Silva Soares and Latouche
 \cite{dasilva2002similarity}.
Similarly, we may define the process observed only during those intervals of time in which
 $X(t)=\max_{0\leq s\leq t}X(s)$
and we call it the process of  \textit{upward records} :
 its generator may be written in terms of the matrix of  first
 return probabilities from below as 
\begin{equation*}
\hat{U}=C_{+}^{-1}A_{++}+C_{+}^{-1}A_{+-}\hat{\Psi}.
\end{equation*}

\section{Erlangization method}\label{SECERL}

The Erlang distribution $T$ with parameters $\nu$ and $L$
may be interpreted as the
time until absorption of $\{{\color{black}\phi(t):t\in\mathbb{R}^{+}\}}$
a Markov chain with $L$ transient stages where the process spends an exponential
time with mean $\nu^{-1}$ in each stage, until an absorbing state.
The infinitesimal generator of 
$\{\phi(t):t\in\mathbb{R}^{+}\}$ is given by $ $the
matrix of order $L$
\[
N=\left[\begin{array}{cccc}
-\nu & \nu\\
 & -\nu & \ddots\\
 &  &  \ddots & \nu\\
 &  &  & -\nu
\end{array}\right]
\]
and $\phi(t)$ is equal to the number of stages that are completed.
The mean of $T$ is $L\nu^{-1}$ and
to approximate a finite time $\theta$,
we choose $\nu=L\theta^{-1}$ for a given $L.$
The variance of $T$ is $\theta^{2}L^{-1}$
and decreases to $0$ as
 $L$ increases to $\infty$.

We construct next the \textit{Erlangized random walk} $\{(X(t),\Phi(t)):t\in\mathbb{R}^{+}\}$
with $\Phi(t)=(\varphi(t),\phi(t))$:
 $\Phi(t)=(i,k)$
 means that at  time $t$,
  the random walk \textit{phase} is $i$
and the Erlang \textit{stage} is $k$,
for $i\in \mathcal{S}$, ${\color{black}k\in\{0,1,...,L-1\}}$.
 We may write the infinitesimal generator of
the \textit{joint phase-and-stage process} $\{\Phi(t):t\in\mathbb{R}^{+}\}$
as follows:
\begin{equation}\label{generatorQ}
Q=\left[\begin{array}{cc|c}
Q_{++} & Q_{+-} & (-N\boldsymbol{1}_{L})\otimes\boldsymbol{1}_{+}\\
Q_{-+} & Q_{--} & (-N\boldsymbol{1}_{L})\otimes\boldsymbol{1}_{-}\\
\hline 0 & 0 & 0
\end{array}\right],
\end{equation}
where $Q_{++}=N\otimes I+I\otimes A_{++},
Q_{+-}=I\otimes A_{+-},
Q_{-+}=I\otimes A_{-+}$,
$Q_{--}=N\otimes I+I\otimes A_{--},$ 
 $\otimes$ denotes
de Kronecker product,
and $\boldsymbol{1}_{+}$,
$\boldsymbol{1}_{-}$
and
$\boldsymbol{1}_{L}$ 
 are a column vectors of ones
of size $|\mathcal{S}_+|$, $|\mathcal{S}_-|$ 
and 
$L$ respectively.
The \textit{Erlangized fluid queue} $\{(Z(t),\Phi(t)):t\in\mathbb{R}^{+}\}$
is constructed similarly.

We shall need here matrices of first return probabilities \textit{before the end of the Erlang period} 
\begin{equation}\label{DefPsiErl}
\boldsymbol{\Psi}_{(i,l)(j,n)}
=
\mathbb{P}
	\left[
	\Phi(\tau_{-}(0))=(j,n)
	|X\left(0\right)=0,
\Phi(0)=(i,l)
\right],
\end{equation}
for $i\in \mathcal{S}_{+}$, $j\in \mathcal{S}_{-}$, and
$0\leq l,n\leq L-1$.
Note that 
$\boldsymbol{\Psi}_{(\cdot,l)(\cdot,n)}$
 does not depend on $l$ and $n$
but only on $n-l$.
It is clear that $\boldsymbol{\Psi}$ 
has an upper triangular block structure,
and that
\begin{equation}\label{BigPSI}
\boldsymbol{\Psi}=\left[\begin{array}{cccccc}
\boldsymbol{\Psi^{(0)}} & \boldsymbol{\Psi^{(1)}} & \boldsymbol{\Psi^{(2)}} &  &  & \boldsymbol{\Psi^{(L-1)}}\\
0 & \boldsymbol{\Psi^{(0)}} & \boldsymbol{\Psi^{(1)}} &  &  & \boldsymbol{\Psi^{(L-2)}}\\
\vdots & 0 & \boldsymbol{\Psi^{(0)}} &  &  & \boldsymbol{\Psi^{(L-3)}}\\
 & \vdots & 0 & \ddots &  & \vdots\\
 &  & \vdots & \ddots\\
 &  &  &  &  & \boldsymbol{\Psi^{(0)}}
\end{array}\right],
\end{equation}
where
\begin{equation}
\boldsymbol{\Psi^{(k)}_{ij}}
=
\mathbb{P}
	\left[
	\Phi(\tau_{-}(0))=(j,k)
	|X\left(0\right)=0,
\Phi(0)=(i,0)
\right].
\end{equation}
 The structure of $\boldsymbol{\hat{\Psi}}$ is similar,
 \begin{equation}\label{DefPsiHatErl}
\boldsymbol{\hat{\Psi}^{(k)}_{ij}}
=
\mathbb{P}
	\left[
	\Phi(\tau_{+}(0))=(j,k)
	|X\left(0\right)=0,
\Phi(0)=(i,0)
\right],
\end{equation}
for $i\in \mathcal{S}_{-}$ , $j\in \mathcal{S}_{+}$,
$0\leq k \leq L-1$.

\begin{Rem}
\rm{
In the analysis of fluid models, one of the matrices
 $\Psi$ and $\hat{\Psi}$
 defined in 
 	(\ref{DefdePsi}) and 
 	(\ref{DefdePsiHat}),
 is stochastic, that is
 	$
\Psi\boldsymbol{1}=\boldsymbol{1},
$
or
$
\hat{\Psi}\boldsymbol{1}=\boldsymbol{1},
$
or both.
Here, the matrices 
$\boldsymbol{\Psi}$
and $\boldsymbol{\hat{\Psi}}$ 
defined in
(\ref{DefPsiErl})
and
	(\ref{DefPsiHatErl})
are sub-stochastic, i.e.
$
\boldsymbol{\Psi1}<\boldsymbol{1}$
and
$
\boldsymbol{\hat{\Psi}1}<\boldsymbol{1}
$,
because they are first return probabilities before $T$.
}
\end{Rem}

The matrices $\boldsymbol{\Psi}$
and $\boldsymbol{\hat{\Psi}}$ 
are recursively determined as follows.
\begin{Thm}\label{RiccatiPsiEq}(Ramaswami et al. {\cite{ramaswami2008erlangization}, Thm 4)}
\begin{enumerate}[(a)]
\item
The matrix $\boldsymbol{\Psi^{(0)}}$ is the minimal nonnegative solution of 
\begin{equation}\label{PsiZeroEq}
\boldsymbol{\Psi^{(0)}}\left\vert C_{-}\right\vert ^{-1}A_{-+}\boldsymbol{\Psi^{(0)}}+C_{+}^{-1}(A_{++}-\nu I)\boldsymbol{\Psi^{(0)}}+\boldsymbol{\Psi^{(0)}}\left\vert C_{-}\right\vert ^{-1}(A_{--}-\nu I)+C_{+}^{-1}A_{+-}=0,
\end{equation}
 and for $1\leq k\leq L-1$,
 $\boldsymbol{\Psi^{(k)}}$
 is the solution of the linear system
\begin{align*}
\boldsymbol{\Psi^{(k)}}\left\vert C_{-}\right\vert ^{-1}\left(A_{--}-\nu I\right)
&+ \sum_{n=0}^{k}\boldsymbol{\Psi^{(n)}}\left\vert C_{-}\right\vert ^{-1}A_{-+}\boldsymbol{\Psi^{(k-n)}}
  \\
+C_{+}^{-1}(A_{++}-\nu I)\boldsymbol{\Psi^{(k)}}
&+ \nu\left(C_{+}^{-1}\boldsymbol{\Psi^{(k-1)}}+\boldsymbol{\Psi^{(k-1)}}\left\vert C_{-}\right\vert ^{-1}\right)  =  0.
\end{align*}
\item 
 The matrix $\boldsymbol{\hat{\Psi}^{(0)}}$ is the minimal nonnegative solution of
\begin{equation}\label{PsiHatEq}
\boldsymbol{\hat{\Psi}^{(0)}} C_{+}^{-1}A_{+-}\boldsymbol{\hat{\Psi}^{(0)}}
+\left\vert C_{-}\right\vert^{-1}(A_{--}-\nu I)\boldsymbol{\hat{\Psi}^{(0)}}
+\boldsymbol{\hat{\Psi}^{(0)}} C_{+} ^{-1}(A_{++}-\nu I)
+\left\vert C_{-}\right\vert^{-1} A_{-+}=0,
\end{equation}
and for $1\leq k\leq L-1$,
\begin{align*} 
\boldsymbol{\hat{\Psi}^{(k)}} C_{+} ^{-1}(A_{++}-\nu I)
&+\sum_{n=0}^{k}
\boldsymbol{\hat{\Psi}^{(n)}} C_{+}^{-1}A_{+-}\boldsymbol{\hat{\Psi}^{(k-n)}}
\\
+\left\vert C_{-}\right\vert^{-1}(A_{--}-\nu I)\boldsymbol{\hat{\Psi}^{(k)}}
&+\nu\left(\left\vert C_{-}\right\vert^{-1}\boldsymbol{\hat{\Psi}^{(k-1)}}
+\boldsymbol{\hat{\Psi}^{(k-1)}}C_{+} ^{-1}\right)  
=  0.
\end{align*}
\end{enumerate}
\hfill $\square$
\end{Thm}

We may define as in the previous Section, the matrices $\mathcal{\boldsymbol{U}}$
and $\mathcal{\hat{\boldsymbol{U}}}$
as the infinitesimal generators of
the processes of the monotone records before $T$.
These
matrices are respectively given by 
\[
\boldsymbol{U}=\left(I\otimes\left|C_{-}\right|\right)^{-1}Q_{--}+\left(I\otimes\left|C_{-}\right|\right)^{-1}Q_{-+}\boldsymbol{\Psi},
\]
for the generator of the downward records and by 
\[
\hat{\boldsymbol{U}}=\left(I\otimes\left|C_{+}\right|\right)^{-1}Q_{++}+\left(I\otimes\left|C_{+}\right|\right)^{-1}Q_{+-}\hat{\boldsymbol{\Psi}},
\]
for the generator of the upwards records.
They
have the same
block-triangular
structure as $\boldsymbol{\Psi}$ and
$\hat{\boldsymbol{\Psi}}$; for instance,
\begin{equation}
\boldsymbol{U}=\left[\begin{array}{cccccc}
\boldsymbol{U}^{(0)} & \boldsymbol{U}{}^{(1)} & \boldsymbol{U}^{(2)} &  &  & \boldsymbol{U}{}^{(L-1)}\\
0 & \boldsymbol{U}^{(0)} & \boldsymbol{U}{}^{(1)} &  &  & \boldsymbol{U}{}^{(L-2)}\\
\vdots & 0 & \boldsymbol{U}{}^{(0)} &  &  & \boldsymbol{U}^{(L-3)}\\
 & \vdots & 0 & \ddots &  & \vdots\\
 &  & \vdots & \ddots\\
 &  &  &  &  & \boldsymbol{U}\mathcal{}^{(0)}
\end{array}\right].
\label{StructureU}
\end{equation}

In the next sections, we determine 
the  probability distribution 
of the level random variable
evaluated at time $T$,
where $T$ has an 
Erlang$(L,L\theta^{-1})$
distribution.
We show in the next lemma that such distribution converge,
as $L \rightarrow \infty$,
to those of the same random variable evaluated at $\theta$.

It will be usefull in the sequel to use the notation 
$\mathcal{Y}_k$ for the sum 
\begin{equation}\label{SumKrv}
\mathcal{Y}_k = \sum_{n=1}^{k}Y_n
\end{equation}
where the random variables $Y_n$,
for $n\geq1$,
are i.i.d.
exponentially distributed random variables
with parameter $\nu$.
In particular, $\mathcal{Y}_L=T$.

We write $\mathbb{P}_{i}^{a}\left[\cdot\right]$
 to denote the conditional probabilities $\mathbb{P}\left[\cdot|X(0)=a,\varphi(0)=i\right]$
for $i\in \mathcal{S}$ and $a\in\mathbb{R}$.
We define the vector of probabilities $\boldsymbol{r}(\cdot,\cdot)$
such that 
\begin{equation}\label{DefdeR}
 r_{i}(y-a,k)= \mathbb{P}^{a}_{i}
 	\left[X(\mathcal{Y}_k)\leq y
	\right],
 \end{equation}
for any $y\in\mathbb{R}$ and $k\in \{1,...,L\}$.
This is also the probability that $X(T)\leq y$
given that at time $0$,
the joint phase-and-stage process
is $(i,L-k)$, i.e. they are $k$ stages left before $T$.

\begin{LEM}
The distribution of the level reached at time
$T\sim Erl(L,L\theta^{-1})$, is such that 
\[
\underset{L\rightarrow\infty}{\lim}r_{i}(x,L)=\mathbb{P}_{i}^{0}\left[X(\theta)\leq x\right],
\]
for
	$i\in \mathcal{S}$,
	$x\in\mathbb{R}$.
\end{LEM}

\begin{proof}

We have to show that for any given $x\in\mathbb{R}$ and
for all $\epsilon>0$,
there exists $k_{0}\in\mathbb{N}$,
such that for all $L\geq k_{0}$,
\[
 \left|r_{i}(x,L)-\mathbb{P}_{i}^{0}\left[X(\theta)\leq x\right]\right|<\epsilon.
\]
To simplify the presentation,
define
 $F(t)=\mathbb{P}_{i}^{0}\left[X(t)\leq x\right]$.
 As $F(\cdot)$ is continuous,
  there exists $\delta>0$, such that 
\begin{equation}
\left|t-\theta\right|<\delta\Rightarrow\left|F(t)-F(\theta)\right|<\frac{\epsilon}{2}.\label{eq:continuity}
\end{equation}
We have
\begin{align}
\left|F(T)-F(\theta)\right| & =  \left|\int_{0}^{\infty}F(t)G(dt)-F\left(\theta\right)\right|,\\
 \intertext{where $G\left(\cdot\right)$ is the distribution function of $T$,}
& \leq
  \int_{0}^{\theta-\delta}\left|F(t)-F\left(\theta\right)\right|G(dt)
  +\int_{\theta-\delta}^{\theta+\delta}\left|F(t)-F\left(\theta\right)\right|G(dt) \nonumber \\
& \quad +  \int_{\theta+\delta}^{\infty}\left|F(t)-F\left(\theta\right)\right|G(dt)
  \label{eq:integralG}
\end{align}
Thanks to (\ref{eq:continuity}), the middle term 
in the right-hand side
of (\ref{eq:integralG}) is bounded by
$\epsilon\mathbb{P}\left[T\in\left[\theta-\delta,\theta+\delta\right]\right]/2
\leq \epsilon/2$.
Define $k_{0}>\frac{2\theta^{2}}{\epsilon\delta^{2}}$ fixed. 
If  $L\geq k_{0}$, then the
other two terms in  (\ref{eq:integralG})  together are bounded
by
\[
\mathbb{P}\left[\left|T-\theta\right|>\delta\right]
<\frac{\theta^{2}}{L\delta^{2}},
\]
by Chebyshev's inequality. We thus obtain the following inequality
\begin{eqnarray*}
\left|r_{i}(x,L)-\mathbb{P}_{i}^{0}\left[X(\theta)\leq x\right]\right| & \leq & \frac{\epsilon}{2}+\frac{\theta^{2}}{L\delta^{2}}
 \quad \leq \quad \epsilon,
\end{eqnarray*}
 which proves the claim.
\end{proof}

Convergence of the other marginal probability distributions
and the joint distributions studied in this paper may be proved by following the same approach. 
In the remainder of the paper, 
we assume that $L$ is fixed.

\section{Time-dependent distribution
of the random walk}\label{TDDRW}
\subsection{Preliminaries}

We consider the fluid model without boundary.
As we saw in 
(\ref{DefdeR}),
the computation of $\boldsymbol{r}$ does not depend on
$a$ and $y$
but on their difference only.
Therefore, to simplify the writing, we suppose that the initial level is $0$.
In Theorem \ref{ThmProbaAbov},
we give recursive equations
for the probabilities of the level being
positive or negative at time $\mathcal{Y}_k$.

The complementary probability of any probability 
$\mathbb{P}\left[\cdot\right]$
is denoted by $\overline{\mathbb{P}}\left[\cdot\right]$,
i.e.
$\overline{\mathbb{P}}\left[\cdot\right]=1-\mathbb{P}\left[\cdot\right]$.
 Denote by $\boldsymbol{h}{(k)}$ the probability
vector that the level is above  $0$ after an Erlang
time period with $k$ stages,
 given that the process starts in level $0$ in
 phase  $i\in \mathcal{S}_{+}$:
\begin{equation}\label{hDefinition}
{h_{i}}{(k)}
=
\mathbb{P}^{0}_{i}[X(\mathcal{Y}_k)>0].
\end{equation}
Similarly,
 $\boldsymbol{\hat{h}}{(k)}$
 is  the probability
vector that the level is below  $0$ after an Erlang
time period with $k$ stages, given that the process starts in
 phase  $i\in \mathcal{S}_{-}$,
in  level $0$ : 
\begin{equation}
{\hat{h}_{i}{(k)}}
=
\mathbb{P}_{i}^{0}[X(\mathcal{Y}_k)<0].
\end{equation}

In the proof of next theorem,
we use the following terminology : we write
that there is a {\textit{down-and-up-crossing}}
when the process, starting in a phase of
	$\mathcal{S}_{+}$ 
at a given level
	$y\in\mathbb{R}$,
returns in a finite time to the same level $y$
in some phase of $\mathcal{S}_{-}$,
crosses that level $y$,
spends some amount of time below $y$ 
and
crosses that level again in some
phase of $\mathcal{S}_{+}$.
Similarly, we define the {\textit{up-and-down-crossing}}.

\begin{Thm}{\label{ThmProbaAbov}} 
For $k=1,$ one has
\begin{align}
\boldsymbol{h}{(1)}&=(I-\boldsymbol{\Psi^{(0)}\hat{\Psi}^{(0)}})^{-1}(\boldsymbol{1}-\boldsymbol{\Psi^{(0)}}\boldsymbol{1}),\label{eq:h1}
\\
\boldsymbol{\hat{h}}{(1)}&=
(I-\boldsymbol{\hat{\Psi}^{(0)}\Psi^{(0)}})^{-1}
(\boldsymbol{1}-\boldsymbol{\hat{\Psi}^{(0)}}\boldsymbol{1}),
\label{eq:hhat1}
\end{align}
and for $k>1$, one has
\begin{align}
\boldsymbol{h}{(k)}
&=(I-\boldsymbol{\Psi}^{(0)}\boldsymbol{\hat{\Psi}}^{(0)})^{-1}
\Big(
\boldsymbol{1}-\sum_{n=0}^{k-1}
\boldsymbol{\Psi}^{(n)}\boldsymbol{1}
+\sum_{{\underset{1 \leq m+n \leq k-1}
{0\leq m,n}}}
\boldsymbol{\boldsymbol{\Psi}}^{(m)}{\boldsymbol{\hat{\Psi}}}^{(n)}\boldsymbol{h}{(k-m-n)}
\Big),
\label{eq:hL}
\\
\boldsymbol{\hat{h}}{(k)}
&=
(I-\boldsymbol{\hat{\Psi}}^{(0)}\boldsymbol{\Psi}^{(0)})^{-1}
\Big(
\boldsymbol{1}
-\sum_{n=0}^{k-1}
\boldsymbol{\hat{\Psi}}^{(n)}\boldsymbol{1}
+\sum_{{\underset{1 \leq m+n \leq k-1}{0\leq m,n}}}\boldsymbol{\boldsymbol{\hat{\Psi}}}^{(m)}
{\boldsymbol{{\Psi}}}^{(n)}
\boldsymbol{\hat{h}}{(k-m-n)}
\Big).
\label{eq:hhatL}
\end{align}
\end{Thm}

\begin{proof}

For $k=1,$ the probability that the level is above the initial level
at the end of an exponential period of time is the sum of the following
probabilities 
\begin{eqnarray*}
\boldsymbol{h}{(1)} & = & \boldsymbol{1}-\boldsymbol{\Psi^{(0)}}\boldsymbol{1}
+\boldsymbol{\Psi^{(0)}}\boldsymbol{\hat{\Psi}^{(0)}}
\boldsymbol{h}{(1)},
\end{eqnarray*}
Indeed,
 $\boldsymbol{1}-\boldsymbol{\Psi^{(0)}}\boldsymbol{1}$
 is the probability that the process
 remains above  level $0$
 without interruption until
 $\mathcal{Y}_1$
 and 
$ \boldsymbol{\Psi^{(0)}}\boldsymbol{\hat{\Psi}^{(0)}}
\boldsymbol{h}{(1)}$
is the probability that the process makes an up-and-down-crossing before  $\mathcal{Y}_1$
and that at the end of the period, the level is above $0$.
As the matrices $\boldsymbol{\Psi}$
and $\boldsymbol{\hat{\Psi}}$ are sub-stochastic,
the inverse of
$(I-\boldsymbol{\Psi^{(0)}\hat{\Psi}^{(0)}})$
exists
and 
(\ref{eq:h1}) is proved.
The proof for
(\ref{eq:hhat1})
is similar.

For $k>1,$ we have
\begin{equation}
\boldsymbol{h}(k)=\left(\boldsymbol{1}-\sum_{n=0}^{k-1}
\boldsymbol{\Psi^{(n)}}
\boldsymbol{1}\right)+\chi^{(k)},
\end{equation}
where the first bracket is the probability that the level remains
above  $y$ during the whole Erlang time period and  $\chi^{(k)}$ denotes the probability that there is at least
one down-and-up-crossing before $T$.
This term $\chi^{(k)}$ can be decomposed as follows
\begin{equation}
\chi^{(k)}=\sum_{\overset{0\leq m,n }{m+n\leq k-1}}\boldsymbol{\Psi^{(m)}\hat{\Psi}^{(n)}}\boldsymbol{h}{(k-m-n)}.
\end{equation}
where the $(m,n)$-th term means that
a down-crossing occurs during the $m$-th stage of
the Erlang, with probability $\boldsymbol{\Psi^{(m)}}$, with $m\in\{0,...,k-1\}$;
then the
level makes an up-crossing
$n$ stages later
 with probability $\boldsymbol{\hat{\Psi}^{(n)}}$
and $n\in\{0,...,k-m-1\}$ ; 
finally, there remains $k-m-n$
stages, so we have to multiply these probabilities by $ $$\boldsymbol{h}{(k-m-n)}$. 
We thus obtain $\boldsymbol{h}{(k)}$ by solving the following equation
\begin{align}
\boldsymbol{h}{(k)}
&=\boldsymbol{1}
-\sum_{n=0}^{k-1}\boldsymbol{\Psi^{(n)}}\boldsymbol{1}
+\sum_{{\overset{0\leq m,n}{{0 \leq m+n\leq k-1}}}}\boldsymbol{\Psi}^{(m)}\boldsymbol{\hat{\Psi}}^{(n)}\boldsymbol{h}{(k-m-n)}. \label{eqlinaer}
\end{align}
Thus (\ref{eq:hL}) is proved
and the proof of 
(\ref{eq:hhatL})
may be done similarly.
\end{proof}

With the next proposition, we show that
once $\boldsymbol{{h}}{(k)}$ is known for all 
$k\in\{1,...,L\}$,
$\boldsymbol{\hat{h}}{(k)}$ is  easily determined, for any $k\in\{1,...,L\}$, and vice versa.

\begin{Prop}\label{Linkhhhat}
One has
\begin{align}
\boldsymbol{\hat{h}}{(k)}&=
\boldsymbol{1}-
\sum_{n=0}^{k-1}
\boldsymbol{\hat{\Psi}^{(n)}}\boldsymbol{{h}}{(k-n)},
\label{HKandHHAT}
\\
\boldsymbol{{h}}{(k)}&=
\boldsymbol{1}-
\sum_{n=0}^{k-1}
\boldsymbol{{\Psi}^{(n)}}\boldsymbol{\hat{{h}}}{(k-n)}.
\label{HHATKandH}
\end{align} 
for all $k$.
\end{Prop}

\begin{proof}
In (\ref{HKandHHAT}),
the sum is the probability that the level is positive at $\mathcal{Y}_k$
given that $\varphi(0)\in\mathcal{S}_-$.
Equation (\ref{HHATKandH}) is also immediate.
\end{proof}

In what follows, we
need to decompose the matrix 
$\exp(\boldsymbol{U} x)$
into sub-blocks, we
 use the following notation
$\boldsymbol{W}_{x}=\exp(\boldsymbol{U} x)$ for 
$x\geq0$
and using the structure
(\ref{StructureU})
 of $\boldsymbol{U}$,
 we write
\begin{equation}
\boldsymbol{W}_{x}
=\left[\begin{array}{cccccc}
\boldsymbol{W}^{(0)}_{x} 
	& \boldsymbol{W}^{(1)}_{x} 
	& \boldsymbol{W}^{(2)}_{x} 
	&
	  & 
	   & \boldsymbol{W}^{(L-1)}_{x}\\
0 
	&\boldsymbol{W}^{(0)}_{x} 
	& \boldsymbol{W}^{(1)}_{x}
	 &  
	 &  
	 & \boldsymbol{W}^{(L-2)}_{x}\\
\vdots 
	& 0 
	& \boldsymbol{W}^{(0)}_{x} 
	 &  
	 &  
	 & \boldsymbol{W}^{(L-3)}_{x}\\
 & \vdots 
 	& 0
	 & \ddots 
	 &  
	 & \vdots\\
 &  
 	& \vdots 
	& \ddots\\
 &  &  &  &  & \boldsymbol{W}^{(0)}_{x} 
\end{array}\right],
\label{StructureW}
\end{equation}
The element
$[\boldsymbol{W}^{(n)}_{x}]_{uv}$
is the conditional probability that the random walk
reaches level $0$
in  phase $v\in  \mathcal{S}_{-}$
during the stage $n$ of the Erlang process,
given that the process starts in level $x$ in phase $u \in  \mathcal{S}_{-}$.
Observe that 
$\boldsymbol{W}^{(0)}_{x} = 
\exp({\boldsymbol{U}}^{(0)}x )$
but the other submatrices have more complex expressions.
We discuss in Section \ref{ILL}
how to determine these in actual practice. 
Similarly, we define the matrix
$\hat{\boldsymbol{W}}_{x}=\exp(\hat{\boldsymbol{U}}x )$,
for $x\geq 0$.
It is also a block-triangular, block-Toeplitz matrix
and we denote by
$\boldsymbol{\hat{W}}^{(n)}_{x}$,
$n=0,\cdots,L-1$,
 the blocks in the first row.

\subsection{Probability distributions}
Recall the definition 
(\ref{DefdeR}) of
$\boldsymbol{r}(\cdot,k)$
as
 the distribution function of the level at time $\mathcal{Y}_k$.
We define the minimum
and the maximum levels reached during an interval, 
\begin{equation}\label{defminMax}
 m(t)=\min\{X(s):0 \leq s \leq t\}
 \text{\quad and \quad }
 M(t)=\max\{X(s):0 \leq s \leq t\},
\end{equation}
for $t\in[0,T]$
and we
define 
 their conditional distribution vectors 
 $\boldsymbol{\eta}(\cdot)$
and
$\boldsymbol{\mu}(\cdot)$
given $X(0)=0$:
\begin{equation}\label{defProbaMinMax}
{\eta}_i(x,k)=\mathbb{P}_{i}^{0}\left[m(\mathcal{Y}_k)\leq x\right]
\text{\quad and \quad }
{\mu}_i(x,k)=\mathbb{P}_{i}^{0}\left[M(\mathcal{Y}_k)\leq x\right]
\end{equation}
for $i \in \mathcal{S}$.
The distributions take different forms according to whether the initial phase is in $\mathcal{S}_{-}$
or in $\mathcal{S}_{+}$
and
we partition all the vectors according to the initial phase
in a manner conformant with  (\ref{GeneratorA}).

\begin{LEM}\label{LemMinMax}
The conditional distribution of $m(\mathcal{Y}_k)$,
  given the initial level and the initial phase, is as follows:
\begin{enumerate}[(a)]
\item
If  $x<0$, then
\begin{align}
\boldsymbol{\eta}_{+}(x,k)
&=\sum_{n=0}^{k-1}
{\boldsymbol{\Psi^{(n)}}
\boldsymbol{\eta}_{-}(x,k-n)},
\label{EqMinReachedP}
\\
\boldsymbol{\eta}_{-}(x,k)
&=\sum_{n=0}^{k-1}
\boldsymbol{W}^{(n)}_{\vert x \vert}
\boldsymbol{1}.
\label{EqMinReachedM}
\end{align}
\item
If $x\geq0$, then
\begin{equation}
\boldsymbol{\eta}_{+}(x,k)
= \boldsymbol{1} 
\qquad
\mbox{ and }
\qquad 
\boldsymbol{\eta}_{-}(x,k)
= \boldsymbol{1}.
\label{EqMinReachedMCase2}
\end{equation}
\end{enumerate}   
  \end{LEM}
  
  \begin{proof}
Take $x<0$.
If $i\in \mathcal{S}_{-}$, we have
\begin{align*}
{\eta}_{i}(x,k)
  =  \mathbb{P}_{i}^{0}\left[\tau_{-}(x)\leq \mathcal{Y}_{k}\right]
  =  \sum_{n=0}^{k-1}{\color{black}\sum_{j\in\mathcal{S}_{-}}\mathbb{P}_{ij}^{0}[\phi\left(\tau_{-}(x)\right)=n]}
 = \boldsymbol{e}_{i}\sum_{n=0}^{k-1}
\boldsymbol{W}^{(n)}_{\vert x \vert}
 \boldsymbol{1},
\end{align*}
where
$\boldsymbol{e_{i}}$
is the vector with
1 in the $i$-th component and zero elsewhere.
If $i\in \mathcal{S}_{+},$
the level  increases at first
and it has to return
 to the initial level $0$ in a phase $u\in \mathcal{S}_{-}$ during
one of the $k$ Erlang stages,
this is given by
 $\boldsymbol{\Psi^{(n)}_{iu}}$
for some $n\in\{0,...,k-1\}$.
Then, starting from a phase in $\mathcal{S}_{-}$, it has to reach below $x$ during 
the remaining $k-n$ exponential steps,
which has probability
$\boldsymbol{\eta}_{-}(x,k-n)$,
which concludes the proof
of 
(\ref{EqMinReachedP})
and
(\ref{EqMinReachedM}).
The proof of 
(\ref{EqMinReachedMCase2})
is immediate.
\end{proof}

  \begin{LEM}
The conditional distribution of $M(\mathcal{Y}_k)$
given the initial level
and the initial phase
is given as follows.
\begin{enumerate}[(a)]
\item
If $x<0$, 
then
\begin{equation}\label{CondMaxNeg}
\boldsymbol{\mu}_{+}(x,k)
=   \boldsymbol{0}
\qquad
\mbox{ and }
\qquad 
\boldsymbol{\mu}_{-}(x,k)
=
  \boldsymbol{0}.
\end{equation}
\item
If  $x\geq0$, then
\begin{align}
\boldsymbol{\mu}_{+}(x,k)
&=
\boldsymbol{1}
-\sum_{n=0}^{k-1}
{\boldsymbol{\hat{W}}}^{(n)}_{x}\boldsymbol{1},
\label{CondMaxPosSPlus}
\\
\boldsymbol{\mu}_{-}(x,k)
&=
\boldsymbol{1}
-
\sum_{n=0}^{k-1}
{\boldsymbol{\hat{\Psi}^{(n)}}}
\overline{\boldsymbol{\mu}}_{+}(x,k-n).
\label{CondMaxPosSMinus}
\end{align}
\end{enumerate}
\end{LEM}

\begin{proof}
To analyze the distribution of $M(\mathcal{Y}_k)$,
we follow an argument similar to
the proof of Lemma \ref{LemMinMax}
and we determine the complementary 
 probability distribution vector with components
$\overline{\mu}_i(x,k)
=\mathbb{P}_{i}^{0}\left[M(\mathcal{Y}_k)> x\right]
$.
This directly leads to 
equations 
	(\ref{CondMaxNeg}),
(\ref{CondMaxPosSPlus}),
(\ref{CondMaxPosSMinus}).
\end{proof}

\begin{Thm}\label{ThmLevelReached}
The conditional distribution of 
$X(\mathcal{Y}_k)$
given that the initial level is $0$ and given  the  initial phase
 is as follows.
\begin{enumerate}[(a)]
\item
If $x\leq0$,  then
\begin{align}
\boldsymbol{r}_{+}(x,k) 
 & =  
\sum_{n=0}^{k-1}
 \boldsymbol{\Psi^{(n)}}
 \boldsymbol{r}_{-}(x,k-n),
  \label{LevelEq2}
  \\
\boldsymbol{r}_{-}(x,k) & = 
\sum_{n=0}^{k-1}
\boldsymbol{W}^{(n)}_{\vert x \vert}
 \boldsymbol{\hat{h}}{(k-n)}.
 \label{LevelEq1}  
\end{align}
\item
If $x>0$, then
\begin{align}
 \boldsymbol{r}_{+}(x,k)
 &= \boldsymbol{1}
 -
 \sum_{n=0}^{k-1}
\hat{\boldsymbol{W}}^{(n)}_{x}
\boldsymbol{h}{(k-n)},
\label{LevelRchdPosUp}
	\\
 \boldsymbol{r}_{-}(x,k)
&= \boldsymbol{1}-
\sum_{n=0}^{k-1}
	\boldsymbol{\hat{\Psi}}^{(n)}
\overline{\boldsymbol{r}}_{+}(x,k-n).
	\label{LevelRchdNegUp}
\end{align}
\end{enumerate} 
\end{Thm}

\begin{proof}
Assume that $x$ is negative
and the initial phase is in $\mathcal{S}_{-}$,
then
to obtain
(\ref{LevelEq1}), we observe that
 the level
has to reach  $x$ during 
some stage $n\in\{0,...,k-1\}$,
 this happens with the probability
${\boldsymbol{W}}^{(n)}_{\vert x \vert}$.
Afterwards, the process has to be below level $x$ at the end of the remainings
 $k-n$ stages left,
 and this occurs
 with probability
 ${\hat{\boldsymbol{h}}{(k-n)}}$,
  given in Theorem \ref{ThmProbaAbov}.

If $i\in \mathcal{S}_{+},$ the level has first  to return the initial
level $0$
during some stage $n\in\{0,...,k-1\}$. 
Then, the process is in a phase of $\mathcal{S}_{-}$
and the argument is the same as before. So   (\ref{LevelEq2})
  is proved.

For $x>0$,
to find
equations
(\ref{LevelRchdPosUp})
and
(\ref{LevelRchdNegUp})
we determine
the probability of the event
$\left[X(\mathcal{Y}_k) > x\right] $
given the initial level $0$
and initial  phase $i$,
 the complement probability of
$r_i(x,k)$,
for which the proof is similar.
\end{proof}

We may use the Erlangization approach to obtain in a simple manner the joint distribution of $X(\mathcal{Y}_{k})$
and the minimum $m(\mathcal{Y}_{k})$
as well as
the joint distribution of 
$X(\mathcal{Y}_{k})$
and the maximum $M(\mathcal{Y}_{k})$.
 We use the notation
 $ \mathcal{P}^{\boldsymbol{\eta}}(x,y,k)$
for the vector with components
 \begin{equation*}
 \mathcal{P}_i^{\boldsymbol{\eta}}(x,y,k)
 =
 \mathbb{P}\left[m(\mathcal{Y}_k)\leq x,X(\mathcal{Y}_k)\leq y\right
 | X(0)=0, \Phi(0)=(i,0)],
 \end{equation*}
and
$ \mathcal{P}^{\boldsymbol{\mu}}(x,y,k)$
for the vector with components
 \begin{equation*}
 \mathcal{P}_i^{\boldsymbol{\mu}}(x,y,k)
 =
 \mathbb{P}\left[M(\mathcal{Y}_k)\leq x,X(\mathcal{Y}_k)\leq y  |
  X(0)=0, \Phi(0)=(i,0)\right],
  \end{equation*}
 where  $i\in\mathcal{S}$,
 $k\in\{1,\cdots,L\}$.

\begin{Thm}\label{ThmJointMinLevel}
The joint
probability distribution of $m(\mathcal{Y}_k)$
and 
$X(\mathcal{Y}_k)$
given the initial level and the initial phase,
is as follows.
\begin{enumerate}[(a)]
\item
If  $x<0$,
then
\begin{align}
\mathcal{P}^{\boldsymbol{\eta}}_{+}(x,y,k)
&=
\sum_{n=0}^{k-1}
\boldsymbol{\Psi^{(n)}}
\mathcal{P}^{\boldsymbol{\eta}}_{-}(x,y,k-n),
  \label{JointProbP}
  \\
  \mathcal{P}^{\boldsymbol{\eta}}_{-}(x,y,k)
&=
\sum_{n=0}^{k-1}
\boldsymbol{W}^{(n)}_{\vert x \vert}
\boldsymbol{r}_{-}(y-x,k-n).
\label{JointProbM}
\end{align}
 \item
If $x\geq0$ then
 \begin{equation}\label{JointMinLevelXPos}
\mathcal{P}^{\boldsymbol{\eta}}_{+}(x,y,k)
 =
 \boldsymbol{r}^{}_{+}(y,k)
 \qquad
\mbox{ and }
\qquad
\mathcal{P}^{\boldsymbol{\eta}}_{-}(x,y,k)
 =
 \boldsymbol{r}^{}_{-}(y,k).
\end{equation}
\end{enumerate}
\end{Thm}

\begin{proof}
Firstly,
assume that $x<0$,
we use the same approach as in the proof of 
(\ref{EqMinReachedM}):
the process has to reach down to level $x$
with probability
$\boldsymbol{W}^{(n)}_{\vert x \vert}$
and at the end of the remaining $k-n$ stages,
takes a value less than $y-x$
and so we obtain (\ref{JointProbM}).

If $\varphi(0)\in\mathcal{S}_{+},$ 
then the level has first to come back to the initial
level $0$
during some stage $n\in\{0,...,k-1\}$,
at which time
the situation is similar to the case 
$\varphi(0)\in \mathcal{S}_{-}$
with $k-n$
Erlang stages left; this gives
(\ref{JointProbP}).

For $x\geq0$,  given that the initial level is $0$, 
$m(\mathcal{Y}_{k})\leq x$ with probability one
and
(\ref{JointMinLevelXPos})
immediately follows.
\end{proof}

\begin{Thm}
The joint probability distribution of 
$M(\mathcal{Y}_k)$
and $X(\mathcal{Y}_k)$
given the initial  level ant the initial phase, is as follows.
\begin{enumerate}[(a)]
\item \label{Case1JointMaxLv}
If $0\leq x$
and $y<x$,
\begin{align}
\mathcal{P}_{+}^{\boldsymbol{\mu}}(x,y,k)
&=
\boldsymbol{r}_{+}(y,k) 
-
\sum_{n=0}^{k-1}
\boldsymbol{\hat{W}}^{(n)}_{x}
\boldsymbol{r}_{+}(y-x,k-n),
\label{JointProbMax}
\\
\mathcal{P}_{-}^{\boldsymbol{\mu}}(x,y,k)
&=
\boldsymbol{r}_{-}(y,k) 
-
\sum_{{\underset{ m+n \leq k-1}
{0\leq m,n}}}
 \hat{\boldsymbol{\Psi}}^{(n)}
\boldsymbol{\hat{W}}^{(m)}_{x}
\boldsymbol{r}_{+}(y-x,k-n-m).
\label{JointProbMaxMinus}
\end{align}
 \item 
If $0>x$ or $y>x$, then
 \begin{equation}\label{Case2JointMaxLv}
\mathcal{P}_{+}^{\boldsymbol{\mu}}(x,y,k)
 =
\boldsymbol{\mu}_{+}(x,k)
 \qquad
 \text{and}
 \qquad
 \mathcal{P}_{-}^{\boldsymbol{\mu}}(x,y,k)
 =
\boldsymbol{\mu}_{-}(x,k).
\end{equation}
\end{enumerate}
\end{Thm}

\begin{proof}

To prove 
(\ref{JointProbMax})
and
(\ref{JointProbMaxMinus}), 
we write
\begin{equation*}
\left[X(\mathcal{Y}_k)\leq y, M(\mathcal{Y}_k)\leq x\right]
=
\left[X(\mathcal{Y}_k)\leq y\right]
\backslash
\left[X(\mathcal{Y}_k)\leq y, M(\mathcal{Y}_k)>x\right],
\end{equation*}
and follow an argument similar 
to
Theorem \ref{ThmJointMinLevel}.
To see (\ref{Case2JointMaxLv}) is obvious.
\end{proof}

\begin{Rem} 
{\bf Link with the bilateral phase-type distribution}

{\rm{
It must be observed that 
$X(T)$ is a particular BPH distribution
which we briefly define.
Suppose that
$\left\{ \zeta(t): t\in\mathbb{R}^{+} \right\} $ 
a Markov process is defined on the state space 
$\mathcal{E} = \{0,1,...,m\}$ 
where states $1,...,m$ are transient and
state $0$ is absorbing.
The infinitesimal generator is ${G}$ with the following structure 
\begin{equation*}
{G}=
\left[\begin{array}{c|c}
D & \boldsymbol{d}\\
 \hline 0 &  0
\end{array}\right],
\end{equation*}
where $D$ is a square matrix of order $m$, $\boldsymbol{d}$ is a column vector
of size $m$ and $\boldsymbol{d}=-D\boldsymbol{1}$. 
The distribution of time $\Delta$ until absorption, defined by 
\begin{equation*}
\Delta=\inf\left\{ t\geq0:D({t})=0\right\}, 
\end{equation*}
is the \textit{phase-type distribution} with representation  $(\boldsymbol{\gamma},D)$,
where $\boldsymbol{\gamma}$ is the initial distribution of $\left\{ \zeta(t)\right\} $
over the transient states.
Define a {\it{Markov modulated fluid model}} $\left\{ (Y(t),\zeta(t)) : t\in \mathbb{R}^{+}\right\} $
such that $Y(0)=0$ and  $Y(\cdot)$ varies linearly as follows 
\begin{equation*}
Y(t)=\int_{0}^{t}e_{\zeta(s)}\mathrm{d}s,
\end{equation*}
where $e_{i}>0$ if$ $ $i\in \mathcal{E}_{+}$ and $e_{i}<0$ if $i\in \mathcal{E}_{-}$,
and
$\mathcal{E}_{+}\cup\mathcal{E}_{-}\cup\{0\}=\mathcal{E}$.
The distribution of $Y(\Delta)$ is the \textit{bilateral phase-type} (BPH) \textit{distribution} with representation $\left(\boldsymbol{\gamma},D,E\right)$, where $E=\mbox{diag}(e_1,...,e_m)$.
It is clear that  $X(T)$ has a BPH distribution where $G$ is given by (\ref{generatorQ}).
The density function $f$ of $Y(\Delta)$ may be written as follows, where $\Psi$, $U$ and $\boldsymbol{h}$
are as in (\ref{BigPSI}), (\ref{StructureU})
and (\ref{hDefinition}), respectively
and similarly for $\hat{\Psi}$, $\hat{U}$
and $\hat{\boldsymbol{h}}$.
We suppose that $\gamma_{0}= 0$ so that the initial state is not the absorbing state $0$.

 \begin{Thm}\label{BPHdistri}
The density function $f$ of $Y(\Delta)$ is given as follows
\begin{numcases}{f(x)=}
\hat{\boldsymbol{p}}  \exp({U}|x|) ({-U}) \hat{\boldsymbol{h}},
&
$x<0$,
\label{BPHNeg}\\
\boldsymbol{p} \exp(\hat{U}x) (-\hat{U}) \boldsymbol{h},
&
$x>0$,
\label{BPHPos}
\end{numcases}
where
$\hat{\boldsymbol{p}}
=\boldsymbol{\gamma}_{+}\Psi+\boldsymbol{\gamma}_{-}$
and
$\boldsymbol{p}=\boldsymbol{\gamma}_{+}+\boldsymbol{\gamma}_{-}\hat{\Psi}$.
\end{Thm}
\begin{proof}
Take $x<0$.
By a similar argument as in Theorem \ref{ThmLevelReached}, we write
$\mathbb{P}[Y(\Delta)\leq x]=\boldsymbol{\gamma}_{+}\Psi \exp({U}|x|) \hat{\boldsymbol{h}}$,
 if the initial state is in $\mathcal{E}_{+}$
 and 
$\mathbb{P}[Y(\Delta)\leq x]=\boldsymbol{\gamma}_{-}
 \exp({U}|x|) \hat{\boldsymbol{h}}$,
if the initial state is in $\mathcal{E}_{-}$. Thus,
\begin{align*}
f(x)
&=
\frac{d}{dx} \left((\boldsymbol{\gamma}_{+}\Psi+\boldsymbol{\gamma}_{-}) \exp({U}|x|) \hat{\boldsymbol{h}}\right)
=
\hat{\boldsymbol{p}}  \exp({U}|x|) ({-U}) \hat{\boldsymbol{h}},
\end{align*}
and this shows (\ref{BPHNeg}).
A similar argument holds for the proof of  (\ref{BPHPos}).
\end{proof}

This expression is equivalent to the one given in 
 Theorem 4.1 in Ahn and Ramaswami \cite{ahn2005bilateral}. To see this, we take the special case $L=1$ in order to simplify the presentation.
 Note that $\boldsymbol{\Psi}^{(0)}$ defined in (\ref{PsiZeroEq}) and $\Psi$ defined here are now identical.
 For $x>0$, 
 \begin{align}
\mathbb{P}[Y(\Delta)>x]
&=  {\boldsymbol{p}}\exp(\hat{U}x)(I-\Psi\hat{\Psi})^{-1}(\boldsymbol{1}-{\Psi} \boldsymbol{1}),
\label{BPHCompar1}\\
\intertext{by
(\ref{LevelRchdPosUp}),
(\ref{LevelRchdNegUp})
and (\ref{eq:h1}).
The calculations in Govorun {\it{et al.}}
 \cite{govorun2013stability},
 page 83 indicate that 
  $(I-\Psi\hat{\Psi}) \hat{U}
  = K (I-\Psi\hat{\Psi})$
 and here the matrix  $(I-\Psi\hat{\Psi})$ is non-singular so that 
 $$\exp{(\hat{U}x)}(I-\Psi\hat{\Psi})^{-1} = (I-\Psi\hat{\Psi})^{-1}\exp({Kx})$$ and
  we may rewrite
  (\ref{BPHCompar1}) as
}
\mathbb{P}[Y(\Delta)>x] &=  \hat{\boldsymbol{p}}(I-\hat{\Psi}\Psi)^{-1}\exp(Kx)(\boldsymbol{1}-\hat{\Psi} \boldsymbol{1}).
\nonumber
\end{align}
By post-multiplying the equation   (\ref{PsiZeroEq})
for $\boldsymbol{\Psi}^{(0)}$ 
 by $\boldsymbol{1}$,
\begin{align*}
E_{+}^{-1}D_{++}\Psi\boldsymbol{1}
+\Psi\left\vert E_{-}\right\vert ^{-1}D_{-+}\Psi\boldsymbol{1}
&=
-(
E_{+}^{-1}D_{+-}\boldsymbol{1}
+\Psi\left\vert E_{-}\right\vert ^{-1}D_{--}\boldsymbol{1}
).
\end{align*}
and  since $\boldsymbol{d}=-D\boldsymbol{1}$,
after algebraic manipulation, we find
$$(\boldsymbol{1}-\Psi \boldsymbol{1})
=(-K)^{-1}[E_{+}^{-1}\boldsymbol{d}_{+}+\Psi E_{-}^{-1}\boldsymbol{d}_{-}],$$
where $K=E_{+}^{-1}D_{++} + \Psi \left\vert E_{-}\right\vert ^{-1} D_{-+}$.
Thus, we obtain
\begin{align}
\mathbb{P}[Y(\Delta)>x]
&=  {\boldsymbol{p}}(I-\hat{\Psi}\Psi)^{-1}\exp(Kx)(-K)^{-1}[E_{+}^{-1}\boldsymbol{d}_{+}+\Psi E_{-}^{-1}\boldsymbol{d}_{-}],
\nonumber
\end{align}
and 
\begin{align}\label{BPHdensityRamPlus}
f(x)={\boldsymbol{p}}(I-\hat{\Psi}\Psi)^{-1}\exp(Kx)[E_{+}^{-1}\boldsymbol{d}_{+}+\Psi E_{-}^{-1}\boldsymbol{d}_{-}],
\quad \quad \text{if } x>0.
\end{align}
With a similar argument, one shows  that
\begin{align}\label{BPHdensityRamMin}
f(x)={\hat{\boldsymbol{p}}}(I-\Psi \hat{\Psi})^{-1}\exp(\hat{K}|x|)[\hat{\Psi}E_{+}^{-1}\boldsymbol{d}_{+}+ E_{-}^{-1}\boldsymbol{d}_{-}],
\quad \quad \text{if } x<0.
\end{align}
where $\hat{K}=|E_{-}|^{-1}D_{--}
+E_{+}^{-1}\hat{\Psi}D_{+-}$.
Equations
(\ref{BPHdensityRamPlus})
and
(\ref{BPHdensityRamMin})
 correspond to the density function given in Theorem 4.1 in Ahn and Ramaswami \cite{ahn2005bilateral}.
}}
\end{Rem}


\section{Time-dependent distribution of fluid queues}
\label{TDDFQ}
\subsection{Boundary effect}
Here, we assume that there is a boundary at the level 0, and
we consider the fluid queue $\{(Z(t),\Phi(t)):t\in\mathbb{R}^{+}\}$
where $Z(t)$ is given by  $(\ref{ZEvolution})$.
To begin with,
we analyze the length of time intervals of time during which the fluid queue remains in level $0$, once it gets there. Specifically,
we determine the
probability that the fluid queue leaves  level $0$
during the Erlang stage $m$,
in a phase  $j\in \mathcal{S}_{+}$, 
given that the process starts in level $0$
in  phase $i\in \mathcal{S}_{-}$, 
that probability
is denoted by 
$\Upsilon^{(m)}_{ij}$,
i.e.
\begin{equation}
\Upsilon^{(m)}_{ij}
	=
	\mathbb{P}_{ij}^{0}
	[ \phi(\beta)=m
	],
	\label{LeavZeroMatrix}
\end{equation}
where $\beta = \inf\{t>0:\varphi(t)\in \mathcal{S}_{+}\}$.
\begin{Thm}
The transition probability matrix of the stages
upon leaving  level $0$,
with components defined in  (\ref{LeavZeroMatrix}),
is given by
\begin{equation}
\Upsilon^{(m)}=\nu^{m}\left(\nu I-A_{--}\right)^{-(m+1)}A_{-+}.\label{eq:probaquitelevel0}
\end{equation}
\end{Thm}

\begin{proof}
We obtain as follows the probability that,
starting from level $0$ and a phase in
$\mathcal{S}_{-}$,
$Z(t)$ remains in level $0$ for $m$
exponential periods, 
and then leaves level $0$
in a phase $j$ of $\mathcal{S}_{+}$
before the start of the  $m+1$-st stage,
\begin{align*}
\Upsilon^{(m)}
	&=
	 \int_{0}^{\infty}
  e^{A_{--}u}
  A_{-+}
   \mathbb{P}[\mathcal{Y}_m
   					< u
					<
					\mathcal{Y}_{m+1}]
    \mathrm{d} u,
   \\
   \intertext{where $\mathcal{Y}_m$ is defined in (\ref{SumKrv}),}
	&=  
	  \int_{0}^{\infty}
  e^{A_{--}u}   A_{-+}
\frac{  (\nu u)^{m}  }{m!}e^{-\nu u}
    \mathrm{d}u,
   \\
&=\nu^{m}  
\left(\nu I-A_{--}\right)^{-(m+1)}
A_{-+}.
\end{align*}
\end{proof}

\subsection{Distribution at time $T$}
In order to determine the distribution
 of the level at time $T$,
we need the distribution under  taboo of level $0$.
This way of analysis is standard in the matrix analytic literature for fluid models.
For instance, see
Ramaswami
\cite{ram99}.

The probability that the fluid queue
is below level $x>0$ 
at time $\mathcal{Y}_k$,
 under  taboo of level $0$,
 given that the process starts in level 
 	$a\in\mathbb{R}^{+}$,
	 in phase $i$,
	 at time $0$,
   is denoted by
\begin{equation*}
{{g}_{i}(a,x,k)}
	=
	\mathbb{P}^{a}[
			Z(\mathcal{Y}_k)<x,
			Z(t)>0,
			\forall t\in (0,\mathcal{Y}_k]
			|\Phi(0)=(i,0)].
\end{equation*}
Here we need to keep track of the initial level
	$Z(0)=a$
 because of the barrier in level $0$:
$\boldsymbol{g}(a,x,k)$
is not equal to $\boldsymbol{g}(0,x-a,k)$ in general.

\begin{LEM}\label{LevelDiff0}
The probability that the fluid queue
is below level $x>0$ 
at time $\mathcal{Y}_k$,
, without returning to level
$0$,
given 
	the initial  level $a\geq0$
	and
	the initial phase,
  is
\begin{align}
	 \boldsymbol{g}_{+}(a,x,k) 
 &=
 \boldsymbol{r}_{+}( x-a ,k)
 -
 \sum_{{\underset{ m+n \leq k-1}
{0\leq m,n}}}
  	\boldsymbol{\Psi}^{(m)}
 	\boldsymbol{W}^{(n)}_{a}
	 \boldsymbol{r}_{-}(x,k-m-n),
	 \\
	 \boldsymbol{g}_{-}(a,x,k) 
 &=
 \boldsymbol{r}_{-}( x-a ,k)
 -
 \sum_{n=0}^{k-1}
 	\boldsymbol{W}^{(n)}_{a}
	 \boldsymbol{r}_{-}(x,k-n) ,
\end{align}
where
$ \boldsymbol{r}(\cdot,\cdot)$
is given
in Theorem \ref{ThmLevelReached}.
\end{LEM}

\begin{proof}
To prove this, 
we use the event decomposition
\begin{equation*}
\left[
Z(\mathcal{Y}_k) < x ,
 Z(t)>0,\forall t\in (0,\mathcal{Y}_k]
 \right]
=
	\left[
		X(\mathcal{Y}_k) < x
\right]
\
\backslash
\
\left[
X(\mathcal{Y}_k) < x,
\tau_{-}(0)< \mathcal{Y}_k
\right],
\end{equation*}
and follow an argument similar 
to the  proof in
Theorem \ref{ThmJointMinLevel}.
\end{proof}

The vector $\boldsymbol{q}(a,x,k)$
of conditional distribution at time   $\mathcal{Y}_k$
is defined like $\boldsymbol{g}$,
with the difference that there is no constraint on 
$Z(t)$ during the interval
$[0,\mathcal{Y}_k ]$:
\begin{equation}\label{defdeqvect}
q_{i}(a,x,k)
=
	 \mathbb{P}_{i}^{a}
	 \left[Z({\mathcal{Y}_k})\leq x
	 \right].
\end{equation}
We give recursive equations
for the computation of this vector.

\begin{Thm}\label{DistribTimeTFF}  
\begin{enumerate}[(a)]
\item[]
\item
For $k\geq1$ and $x>0$, one has
\begin{align}
\boldsymbol{q}_{+}(0,x,k)
  &= 
  \left(I-\boldsymbol{\Psi^{(0)}}\Upsilon^{(0)}\right)^{-1}
\Big( \boldsymbol{g}_{+}(0,x,k) \nonumber
\\
&\quad +
  \sum_{{\underset{1 \leq m+n \leq k-1}{0\leq m,n}}}
\boldsymbol{\Psi}^{(n)}
\Upsilon^{(m)}
\boldsymbol{q}_{+}(0,x,k-n-m)
  \Big),
  \label{qPlus}
  \\
 \boldsymbol{q}_{-}(0,x,k) 
 &=
 \sum_{n=0}^{k-1}
 \Upsilon^{(n)}
 \boldsymbol{q}_{+}(0,x,k-n),
  \label{qMinus}
\end{align}
where
$\boldsymbol{g}_{+}(0,x,\cdot)$
is given in Lemma \ref{LevelDiff0}.

\item 
For $k\geq1$, for  $a,x>0$,
one has
\begin{align}
\boldsymbol{q}_{+}(a,x,k) 
 &=
 \boldsymbol{g}_{+}(a,x,k)
 +
 \sum_{{\underset{ m+n \leq k-1}
{0\leq m,n}}}
  \boldsymbol{\Psi}^{(n)}
 \boldsymbol{W}^{(m)}_{a}
 \boldsymbol{q}_{-}(0,x,k-n-m),
 \\
 \boldsymbol{q}_{-}(a,x,k) 
 &=
 \boldsymbol{g}_{-}(a,x,k)
 +
 \sum_{n=0}^{k-1}
 \boldsymbol{W}^{(n)}_{a}
 \boldsymbol{q}_{-}(0,x,k-n).
\end{align}
\end{enumerate}
\end{Thm}

\begin{proof}
There are two
ways to be in $[0,x]$ at time $\mathcal{Y}_k$,
 starting from $0$ in a phase of $\mathcal{S}_{+}$: 
either  this occurs 
without returning  to the level $0$ during the period 
$(0,\mathcal{Y}_k]$,
or
 the process first comes back to $0$,
 stays there for some amount of time,
leaves level $0$
and is in $[0,x]$
at the end of the  remaining number of stages.
This gives the decomposition
\begin{align}
\boldsymbol{q}_{+}(0,x,k)= \boldsymbol{g}_{+}(0,x,k)+\sum_{n=0}^{k-1}
\sum_{m=0}^{k-n-1}
\boldsymbol{\Psi}^{(n)}
\Upsilon^{(m)}
\boldsymbol{q}_{+}(0,x,k-n-m).
\end{align}
which leads to the equation (\ref{qPlus}).

Equation (\ref{qMinus}) is obtained by noting that
the process first remain in the level $0$ for $n$ Erlang stages and it leaves  level $0$
in a phase of $\mathcal{S}_{+}$.

We deduce easily the general case for the distribution
$\boldsymbol{q}(a,x,k)$, for 
$a>0$:
we decompose this into the probability that $X(\mathcal{Y}_k)\leq x$ without returning to $0$ and the probability that the process returns to $0$ before $\mathcal{Y}_k$. 
\end{proof}

Note that when $k=1$, the equations simplify. In particular, for instance, the second term in the big bracket in (\ref{qPlus}) disappear as the sum is  empty.

\subsection{Distributions of the minimum and the maximum}

Define the minimum
and the maximum level reached during an interval
for the model bounded at $0$: 
\begin{equation}\label{defminMaxFluid}
 m_0(t)=\min\{Z(v):0\leq v \leq t\}
 \text{\quad and \quad }
 M_0(t)=\max\{Z(v):0\leq v \leq t\}.
\end{equation}

We denote by 
$\boldsymbol{\rho}(a,x,k)$
and by
$\boldsymbol{\delta}(a,x,k)$
the vectors of the  conditional distribution function of the minimum
and the maximum reached during the Erlang period:
\begin{equation}\label{defProbaMinMaxFluid}
{\rho}_i(a,x,k)=\mathbb{P}_{i}^{a}\left[m_0(\mathcal{Y}_k)\leq x\right]
\text{\quad and \quad }
{\delta}_i(a,x,k)=\mathbb{P}_{i}^{a}\left[M_0(\mathcal{Y}_k)\leq x\right]
\end{equation}
for $i \in \mathcal{S}$.
\begin{LEM}
The conditional distribution of the minimum level reached
by the fluid queue  during
the Erlang horizon period, given
the initial level $a$
and the initial phase, is
\begin{align}
\boldsymbol{\rho}_+(a,x,k)
=\boldsymbol{\eta}_+(a-x,k)
\qquad
\text{and}
\qquad
\boldsymbol{\rho}_-(a,x,k)
=\boldsymbol{\eta}_-(a-x,k),
\end{align}
for $a,x\geq0$ and
where $\boldsymbol{\eta}(\cdot,\cdot)$
is given in the Lemma
\ref{LemMinMax}.
\end{LEM}

\begin{proof}
The proof  is immediate as it is clear that for $x\geq0$,
the event 
$[m(t)>x]$
and $[m_0(t)>x]$
are identical.
\end{proof}

The analysis of $M_{0}(\mathcal{Y}_k)$
is more involved, due to the barrier at zero,
and we need a new set of first passage 
 first passage probability matrices 
between the levels $0$ and $x$.
Denote by $\boldsymbol{\Lambda}^{(k)}_{x}$ the matrix with components
\begin{equation*}
	(\boldsymbol{\Lambda}^{(k)}_{x})_{ij}
	=
	\mathbb{P}_i^{0}
	[
	\tau_{+}(x)
	<
	\tau_{-}(0),
		\tau_{+}(x)
		<
	\mathcal{Y}_k,
	\Phi(	\tau_{+}(x)) = (j,k)
	],
	\end{equation*}
	for $i,j\in \mathcal{S}_{+}$,
	$0\leq k \leq L-1$, that is
$(\boldsymbol{\Lambda}^{(k)}_{x})_{ij}$ is the probability,
starting from $(0,i)$,
 of reaching level $x>0$ in phase $j$,
  during the stage $k$ of the Erlang,
under taboo of  level zero.
Symmetrically,
denote by $\boldsymbol{\Psi}^{(k)}_{x}$ the matrix with components
\begin{equation*}
	(\boldsymbol{\Psi}^{(k)}_{x})_{iu}
	=
	\mathbb{P}_i^{0}
	[
	\tau_{-}(0)
	<
	\tau_{+}(x),
		\tau_{-}(0)
		<
		\mathcal{Y}_k,
	\Phi(\tau_{-}(0)) = (u,k)
	]
	\end{equation*}
	with $i\in \mathcal{S}_{+}, u\in \mathcal{S}_{-}$,
$0\leq k \leq L-1$,
that is
$(\boldsymbol{\Psi}^{(k)}_{x})_{ij}$ is the probability,
starting from $(0,i)$,
of returning to the level $0$
during the stage $k$ of the Erlang
without having reached level $x$.

Furthermore,
define the matrices $\hat{\boldsymbol{\Lambda}}^{(k)}_{x}$ 
and 
$\hat{\boldsymbol{\Psi}}^{(k)}_{x}$
 with  components
\begin{align*}
	(\hat{\boldsymbol{\Lambda}}_{x}^{(k)})_{ij}
	&=
	\mathbb{P}_i^{x}
	[
	\tau_{-}(0)
	<
	\tau_{+}(x),
		\tau_{-}(0)
		<
		\mathcal{Y}_k,
	\Phi(	\tau_{-}(0)) = (j,k)
	],
	\\
	(\hat{\boldsymbol{\Psi}}^{(k)}_{x})_{iu}
	&=
	\mathbb{P}_i^{x}
	[
	\tau_{+}(x)
	<
	\tau_{-}(0),
		\tau_{+}(x)
		<
		\mathcal{Y}_k,
	\Phi(\tau_{+}(x)) = (u,k)
	],
	\end{align*}
with $i,j\in \mathcal{S}_{-},u\in \mathcal{S}_{+}$,
$0\leq k \leq L-1$.

We denote by 
$\boldsymbol{\Omega}_x=
\boldsymbol{\Psi}^{(0)}
\boldsymbol{W}^{(0)}_{x}
$,
$\boldsymbol{\Omega}^{(n,k)}_x=
\boldsymbol{\Psi}^{(n)}
\boldsymbol{W}^{(k-n)}_{x}
$,
$\hat{\boldsymbol{\Omega}}_x=
\boldsymbol{\hat{\Psi}^{(0)}}
\boldsymbol{\hat{W}}^{(0)}_{x}
$,
and 
$\hat{\boldsymbol{\Omega}}^{(n,k)}_x=
\boldsymbol{\hat{\Psi}^{(n)}}
\boldsymbol{\hat{W}}^{(k-n)}_{x}
$.
The matrices
$\boldsymbol{\Lambda}^{(k)}_{x}$
and 
$\hat{\boldsymbol{{\Lambda}}}^{(k)}_{x}$
can be expressed according to 
the matrices 
$\boldsymbol{\Psi}^{(k)}_{x}
$
and 
 $\hat{\boldsymbol{\Psi}}^{(k)}_{x}$
which are computed recursively as follows. 
 
 \begin{Thm}
 Let $x>0$.
\begin{enumerate}[(a)]
\item
 The matrices
 $ {\boldsymbol{\Lambda}}^{(0)}_{x}$
 and 
  $ \hat{{\boldsymbol{\Lambda}}}^{(0)}_{x}$
are
given by
\begin{align}
  {\boldsymbol{\Lambda}}^{(0)}_{x}
 &=
   \hat{\boldsymbol{W}}^{(0)}_{x}
   - 
{\boldsymbol{\Psi}}^{(0)}_{x}
  \hat{\boldsymbol{\Omega}}_x
  {\label{LambdaZero}}
\\
  \hat{\boldsymbol{\Lambda}}^{(0)}_{x}
     &=\boldsymbol{W}^{(0)}_{x}
   -		\hat{\boldsymbol{\Psi}}^{(0)}_{x}
	   \boldsymbol{\Omega}_x
	   {\label{LambdaHatZero}}
	 \end{align}	 
where 
\begin{align}\label{PsiZeroB}
{}\boldsymbol{\Psi}^{(0)}_{x}
&=
\left(\boldsymbol{\Psi}^{(0)}
-\hat{\boldsymbol{W}}^{(0)}_{x}
\boldsymbol{\Omega}_x\right)
(I
-
\hat{\boldsymbol{\Omega}}_x
\boldsymbol{\Omega}_x
)^{-1},
\\
\hat{\boldsymbol{\Psi}}^{(0)}_{x}
&=
\left(
\hat{\boldsymbol{\Psi}}^{(0)}-
\boldsymbol{W}^{(0)}_{x}
\hat{\boldsymbol{\Omega}}_x
\right)
(I-
\boldsymbol{\Omega}_x
\hat{\boldsymbol{\Omega}}_x
)^{-1}.
\label{PsiZeroHatB}
\end{align}
\item For $1\leq k\leq L-1$,
the matrices
$ {\boldsymbol{\Lambda}}^{(k)}_{x}$
and
 $\hat{{\boldsymbol{\Lambda}}}^{(k)}_{x}$
 are given by
	 \begin{align}
  {\boldsymbol{\Lambda}}^{(k)}_{x}
  &=
  \hat{\boldsymbol{W}}^{(k)}_{x}
  -  \sum_{{\underset{0 \leq m+n \leq k}
{0\leq m,n}}}
 \boldsymbol{\Psi}^{(m)}_{x}
  \hat{\boldsymbol{\Omega}}^{(n,k-m)}_x
  \label{LambdaK_b}
  \\
  {\boldsymbol{\hat{\Lambda}}}^{(k)}_{x}
  &=
  {\boldsymbol{W}}^{(k)}_{x}
  -  \sum_{{\underset{0 \leq m+n \leq k}{0\leq m,n}}}
\boldsymbol{\hat{\Psi}}^{(m)}_{x}
  \boldsymbol{\Omega}^{(n,k-m)}_x
    \label{LambdaHatK_b}
	 \end{align}
where 
the matrices
$\boldsymbol{\Psi}^{(k)}_{x} $
and
$\hat{\boldsymbol{\Psi}}^{(k)}_{x} $
are given by
\begin{align}
\boldsymbol{\Psi}^{(k)}_{x}
 = 
\Bigg(  \boldsymbol{\Psi}^{(k)}
& -
\sum_{\substack{
         0\leq m \leq k-1\\
            0\leq n\leq k-m 
            }}
{\boldsymbol{\Lambda}}^{(m)}_{x}
{\boldsymbol{\Omega}}^{(n,k-m)}_{x}
\nonumber
\\
& -
\Big[  \hat{ \boldsymbol{{W}}}^{(k)}_{x}
			+
			\sum_{\substack
					{
       					  0\leq m \leq k-1\\
         		   			0\leq n\leq k-m 
          					  }}
		\boldsymbol{\Psi}^{(m)}_{x}
		\hat{ \boldsymbol{\Omega}}^{(n,k-m)}_{x}
\Big]
	{\boldsymbol{\Omega}}_{x}
\Bigg)
(I-
\hat{\boldsymbol{\Omega}}_x
{\boldsymbol{\Omega}}_x
)^{-1}
\label{Psik_b}
\\
\hat{\boldsymbol{\Psi}}^{(k)}_{x} 
= 
	\Bigg(  \hat{\boldsymbol{\Psi}}^{(k)}
&-
\sum_{\substack{
         0\leq m \leq k-1\\
            0\leq n\leq k-m 
            }}
            {\boldsymbol{\hat{\Lambda}}}^{(m)}_{x}
            \hat{ \boldsymbol{\Omega}}^{(n,k-m)}_{x}
            \nonumber
            \\
&-
\Big[ \boldsymbol{{W}}^{(k)}_{x}
+
\sum_{\substack{
         0\leq m \leq k-1\\
            0\leq n\leq k-m 
            }}
{\hat{\boldsymbol{\Psi}}^{(m)}_{x}}
\boldsymbol{\Omega}^{(n,k-m)}_{x}
\Big]
\hat{\boldsymbol{\Omega}}_x
	\Bigg)
	 (I-
		{\boldsymbol{\Omega}}_x
		\hat{\boldsymbol{\Omega}}_x
		)^{-1}
\end{align}
\end{enumerate}
\end{Thm}

 \begin{proof}
 The matrix  $  \hat{\boldsymbol{W}}^{(0)}_{x}$
can be decomposed as the sum of  the probability 
that the process reaches level $x$
before returning to the level $0$
and the probability that the process returns to the level $0$ before reaching the level $x$.
It can thus be written as
 \begin{equation*}
  \hat{\boldsymbol{W}}^{(0)}_{x} =
  {\boldsymbol{\Lambda}}^{(0)}_{x}
  +
{\boldsymbol{\Psi}}^{(0)}  _{x}
  \hat{{\boldsymbol{\Psi}}}^{(0)}
   \hat{\boldsymbol{W}}^{(0)}_{x},
	 \end{equation*}
and therefore, equation 
 (\ref{LambdaZero})
 is proved.

 Recall that the matrix $\boldsymbol{\Psi}^{(0)}$
 is such that each component
 $\boldsymbol{\Psi}^{(0)}_{ij}$ is the probability,
 starting from $(0,i)$, 
 of returning to the level $0$
 before one exponential stage is finished, 
  with $i\in\mathcal{S}_{+}$
 and  $j\in\mathcal{S}_{-}$:
 it can thus be decomposed as the sum 
 \begin{equation}\label{ProofPsiZerob}
\boldsymbol{\Psi}^{(0)}
={}\boldsymbol{\Psi}^{(0)}_{x}
 +\boldsymbol{\Lambda}^{(0)}_{x}
\boldsymbol{\Psi}^{(0)}
  \boldsymbol{W}^{(0)}_{x},
\end{equation}
where the first term is the probability
 that the process returns to the level $0$ before crossing the level $x$,
 and the second is the probability
 that the process first visit the level $x$
 before it returns to the level $0$.
 In the second case,
 the process must return to level $x$ in a phase of 
 $\mathcal{S}_{-}$ which justifies the factor 
 $\boldsymbol{\Psi}^{(0)}$,
 before eventually going down to level $0$.
 Using
 (\ref{LambdaZero})
 in the equation
 (\ref{ProofPsiZerob})
leads to
 (\ref{PsiZeroB}).
 
 We prove ({\ref{LambdaHatZero}})
 and (\ref{PsiZeroHatB})
 in a similar manner.
For $1\leq k\leq L-1$,
we need to keep track of the stages of the Erlang and readily  obtain the expressions 
(\ref{LambdaK_b})
and
(\ref{LambdaHatK_b})
for 
the matrices
$ {\boldsymbol{\Lambda}}^{(k)}_{x}$
and
 $\hat{{\boldsymbol{\Lambda}}}^{(k)}_{x}$.

The matrix
$\boldsymbol{\Psi}^{(k)}_{x}
 $
is obtained by noting that 
\begin{equation}
\boldsymbol{\Psi}^{(k)}
=\boldsymbol{\Psi}^{(k)}_{x}
 +\sum_{{\underset{0 \leq m+n \leq k}{0\leq m,n}}}
 \boldsymbol{\Lambda}^{(m)}_{x}
\boldsymbol{\Psi}^{(n)}
  \boldsymbol{W}^{(k-n-m)}_{x},
\end{equation}
and by replacing
$ \boldsymbol{\Lambda}^{(m)}_{x}$
by its
expression given in 
 (\ref{LambdaK_b}).
 By reorganizing the terms we find the equation
 (\ref{Psik_b}).
  \end{proof}

Now we can give the distribution of the maximum
level reached during the Erlang interval,
given the initial level,
we follow the (by now) familiar accounting of the stages of the Erlang interval.
\begin{Thm}
\begin{enumerate}[(a)]
\item Take $a=0$. 
For $k\geq1$
and $x\geq0$,
\begin{align}\label{DeltaComplPlusEq}
\overline{\boldsymbol{\delta}}_{+}(0,x,k)	
	=
	&\left(I-
	\boldsymbol{\Psi}_{x}^{(0)}\Upsilon^{(0)}\right)^{-1}
	\nonumber
	\\
&
	\Bigg(
	\sum_{n=0}^{k-1}
	\boldsymbol{\Lambda}^{(n)}_{x}
		\boldsymbol{1}
	+
	\sum_{\overset{0\leq n,l}{1\leq n+l\leq k-1}}
	\boldsymbol{\Psi}^{(n)}_{x}
	\Upsilon^{(l)}
	\overline{\boldsymbol{\delta}}_{+}(0,x,k-n-l)
	\Bigg).
\end{align}

For $k\geq1$, the conditional distribution of $M_0(\mathcal{Y}_k)$
given the initial level $a=0$, 
is given
by
\begin{align}
\boldsymbol{\delta}_{+}(0,x,k)
&=
	\boldsymbol{1}
	-  
	\overline{\boldsymbol{\delta}}_{+}(0,x,k)
		,
		\label{DeltaZeroPlus}
\\
\boldsymbol{\delta}_{-}(0,x,k)
&=
	\boldsymbol{1}
	- 
	\sum_{n=0}^{k-1}
	\Upsilon^{(n)}
	\overline{\boldsymbol{\delta}}_{+}(0,x,k-n)	
		\label{DeltaZeroMoins}
\end{align}

\item
Take $a>0$.
For $k\geq1$,
\begin{align}
\overline{\boldsymbol{\delta}}_{+}(a,x,k)
=
&\left(I- \boldsymbol{\Psi}^{(0)}_{x-a}
\boldsymbol{\hat{\Psi}}^{(0)}_{a}\right)^{-1}
\nonumber
\\
&
  \Bigg[  \sum_{n=0}^{k-1}
 	\boldsymbol{\Lambda}^{(n)}_{x-a}\boldsymbol{1}
 	+  \boldsymbol{\Psi}^{(0)}_{x-a}
	\hat{\boldsymbol{\Lambda}}^{(0)}_{a}
 	\overline{\boldsymbol{\delta}}_{-}(0,x,k)
	\nonumber
	\\
 &  +
 	\sum_{\overset{0\leq n,l}{1\leq n+l\leq k-1}}
 \boldsymbol{\Psi}	^{(n)}_{x-a}
	 \Big(
		\boldsymbol{\hat{\Psi}}^{(l)}_{a}
		\overline{ \boldsymbol{\delta}}_{+}(a,x,k-n-l)
		 \nonumber
		 \\
		 & \quad\quad\quad\quad\quad\quad\quad\quad
		 \quad\quad
		 +\;
		 \hat{\boldsymbol{\Lambda}}^{(l)}_{a}\;
		 \overline{\boldsymbol{\delta}}_{-}(0,x,k-n-l)
	 \Big)
	 \Bigg].
	 \label{EqDeltaPlusComplGen}
\end{align}

For $k\geq1$, the conditional distribution of $M_0(\mathcal{Y}_k)$
given the initial level $a>0$, 
is given for $x\geq a$,
by
\begin{align} 
\boldsymbol{\delta}_{+}(a,x,k) 
& =  \boldsymbol{1}-
\overline{\boldsymbol{\delta}}_{+}(a,x,k),
\label{DeltaPlusGivenA}
\\
\boldsymbol{\delta}_{-} (a,x,k) 
&= \boldsymbol{1}-
	 \sum_{n=0}^{k-1} 
	 \Big(
	\boldsymbol{\hat{\Psi}}^{(n)}_{a}
	\overline{  \boldsymbol{\delta}}_{+}s(a,x,k-n) 
	  + \hat{\boldsymbol{\Lambda}}^{(n)}_{a}
	  \overline{\boldsymbol{\delta}}_{-}(0,x,k-n) 
	  \Big),
	  \label{DeltaMoinsGivenA}
\end{align}

and for $a>x$, by
\begin{align}
\boldsymbol{\delta}_+(a,x,k)
=
\boldsymbol{0},
\qquad
\text{and}
\qquad
\boldsymbol{\delta}_-(a,x,k)
=
\boldsymbol{0}.
\label{MaxEvident}
\end{align}
\end{enumerate}
\end{Thm}

\begin{proof}
For $k\geq1$, $x\geq0$,
any $i$,
$\overline{\boldsymbol{\delta}}_{i}(0,x,k)	$
is the probability that the process reaches level $x$
before time $\mathcal{Y}_k$.
Thus,
\begin{align}
\overline{\boldsymbol{\delta}}_{+}(0,x,k)	
	&=
	\sum_{n=0}^{k-1}
	\boldsymbol{\Lambda}^{(n)}_{x}
	\boldsymbol{1}
	+
	\sum_{n=0}^{k-1}
	\boldsymbol{\Psi}^{(n)}_{x}
	\sum_{l=0}^{k-n-1}
	\Upsilon^{(l)}
	\overline{\boldsymbol{\delta}}_{+}(0,x,k-n-l).
	\label{MaxFluid}
\end{align}
To show this :
 the 
first term is the probability that the 
level reaches $x$ without returning to level $0$;
the second term is the probability that the process first comes back in level $0$ under taboo of level $x$,
during the $n$-th stage,
with probability
	$\boldsymbol{\Psi}^{(n)}_{x}$,
		then the level stays in level $0$
	for $l$ Exponential intervals,
	with probability
	$\Upsilon^{(l)}$
	at which time
	the process is in the initial situation, but with $k-n-l$ stages left:
	this is given by the probability
$\overline{\boldsymbol{\delta}}_{+}(0,x,k-n-l)$.	
We reorganize the terms in equation
(\ref{MaxFluid})
and obtain
(\ref{DeltaComplPlusEq}),
as $(I-	\boldsymbol{\Psi}^{(0)}_{x}
\Upsilon^{(0)})
$
is non-singular.

For $k\geq1$,
the probability vector 
$\boldsymbol{\delta}_{+}(0,x,k)$
is trivially the complement probability 
vector
of (\ref{DeltaComplPlusEq}), so we have
Equation (\ref{DeltaZeroPlus}).

For $k\geq1$, equation (\ref{DeltaZeroMoins})
 is obtained by noting that for $i\in\mathcal{S}_-$,  the complementary probability 
 of
$ \boldsymbol{\delta}_{-}(0,x,k)$
is  given by
 \begin{equation*}
\overline{\boldsymbol{\delta}}_{-}(0,x,k)
=
\sum_{n=0}^{k-1}
	\Upsilon^{(n)}
	\overline{\boldsymbol{\delta}}_{+}(0,x,k-n):
\end{equation*}
it is the probability $\Upsilon^{(n)}$ that
  the level first  leaves $0$ in one of the $k$ Erlang stages, times the probability 
  $\overline{\boldsymbol{\delta}}_{+}(0,x,k-n)$
   that the level reaches the level $x$, in one of the $k-n$ stages left.
  So, equation (\ref{DeltaZeroMoins}) is proved.

  Take $a>0$, and $n\geq1$.
  For $x\geq a$,
  $\overline{\boldsymbol{\delta}}_{+}(a,x,n)	$ can be decomposed in a manner similar to  (\ref{MaxFluid})   as follows,
  \begin{align}
\overline{\boldsymbol{\delta}}_{+}(a,x,k)	
	=
	&\sum_{n=0}^{k-1}
	\boldsymbol{\Lambda}^{(n)}_{x-a}
	\nonumber
	+\sum_{n=0}^{k-1}\sum_{l=0}^{k-n-1}
	\boldsymbol{\Psi}^{(n)}_{x-a}
	\hat{\boldsymbol{\Lambda}}^{(l)}_{a}
		 \overline{\boldsymbol{\delta}}_{-}(0,x,k-n-l)
		 \nonumber
		 \\
		 &+
		 \sum_{n=0}^{k-1}
		  \sum_{l=0}^{k-n-1}
	\boldsymbol{\Psi}^{(n)}_{x-a} 
	\hat{	\boldsymbol{\Psi}}^{(l)}_a 
			 \overline{\boldsymbol{\delta}}_{+}(a,x,k-n-l).
	\label{EqDeltaPlusComplProof}
\end{align}
Indeed, we have three cases:

\noindent\textbf{Case 1.} 
The process avoids $a$:
it reaches $x$ in the $n$-th stage without returning to level $a$.

\noindent\textbf{Case 2.} 
The process returns to $a$ and  drops to $0$: 
the process first comes back in level $a$  with the taboo of level $x$, with probability
$\boldsymbol{\Psi}^{(n)}_{x-a}$,
then it reaches level $0$ in $l$, one of the $k-n$ stages left, before returning to $a$,
with probability ${\hat{\boldsymbol{\Lambda}}}^{(l)}_{a}$,
and finally the process reaches level $x$,
starting from level $0$,  with probability
$ \overline{\boldsymbol{\delta}}_{-}(0,x,k-n-l)$.

\noindent\textbf{Case 3.} 
The process returns twice to $a$
avoiding $x$ and $0$ :
 first it comes back to level $a$ from above in the $n$-th stage under taboo of level $x$, with probability 
$\boldsymbol{\Psi}^{(n)}_{x-a}$,
then it comes back to level $a$ from below
in the $l$-th stage, 
with the taboo of level $0$,
with probability
$\hat{	\boldsymbol{\Psi}}^{(l)}_a$
and 
finally the level reaches $x$ in one of the $k-n-l$ remaining stages,
with probability
$ \overline{\boldsymbol{\delta}}_{+}(a,x,k-n-l)$.

After algebraic manipulation
of the terms in (\ref{EqDeltaPlusComplProof})
we obtain 
(\ref{EqDeltaPlusComplGen})
since
 $(I- \boldsymbol{\Psi}^{(0)}_{x-a}
\boldsymbol{\hat{\Psi}}^{(0)}_{a})$
is non-singular.

For $k\geq1$ and $x>a$,
$\boldsymbol{\delta}_{+}(a,x,k)$
given in Equation (\ref{DeltaPlusGivenA})
is immediate.

To show (\ref{DeltaMoinsGivenA}),
we note that
  the complementary probability vector
 of
$ \boldsymbol{\delta}_{-}(a,x,k)$,
can be decomposed as follows,
\begin{equation}
 \overline{\boldsymbol{\delta}}_{-}(a,x,k)
 =\sum_{n=0}^{k-1}
 \boldsymbol{\hat{\Psi}}^{(n)}_{a}
	 \overline{ \boldsymbol{\delta}}_{+}(a,x,k-n) 
	  +
	  \sum_{n=0}^{k-1}
	  \hat{\boldsymbol{\Lambda}}^{(n)}_{a}
	  \overline{\boldsymbol{\delta}}_{-}(0,x,k-n).
 \end{equation}
 To see this:
 	the first term is the probability that the process first returns to the level $a$  in stage $l$, before it reaches level $0$, i.e.
$\boldsymbol{\hat{\Psi}}^{(n)}_{a}$,
multiplied by the probability that the level reaches level $x$ during one of the $k-n$ stages left, i.e.
$\overline{\boldsymbol{\delta}}_{+}(a,x,k-n)$;
	the second term is the probability that the process first reaches the level $0$ before returning to level $a$, this is 
	$\hat{\boldsymbol{\Lambda}}^{(n)}_{a}$, multiplied by the probability that the process reaches level $x$, given that the initial level is $0$,
	the phase is in $\mathcal{S}_-$ and that there are $k-n$ stages left,
	i.e. $\overline{\boldsymbol{\delta}}_{-}(0,x,k-n)$.

If $a>x$, then obviously 
$\boldsymbol{\delta}(a,x,k)=0$, as it is not possible that the maximum of the process
should be less
 or equal to $x$ if the initial level is already greater then $x$. This gives (\ref{MaxEvident}).
\end{proof}

Note that equations 
(\ref{DeltaComplPlusEq})
and (\ref{EqDeltaPlusComplGen})
simplify a lot when $k=1$ 
because in this case, they contain sums on empty sets.

\subsection{Joint distribution of $m_0(T)$ and $Z(T)$}

Denote the joint distribution vector of $m_0(\mathcal{Y}_{k})$ and  $Z(\mathcal{Y}_{k})$ by
 $ \mathcal{P}^{\boldsymbol{\rho,q}}(a,x,y,k)$,
 which is the vector with components
 \begin{equation*}
 \mathcal{P}_i^{\boldsymbol{\rho,q}}(a,x,y,k)
 =
 \mathbb{P}_{i}^{a}\left[m_0(\mathcal{Y}_k)\leq x,Z(\mathcal{Y}_k)\leq y\right],
 \end{equation*}
 where  $i\in\mathcal{S}$,
 $1\leq k \leq L$,
 $a,x\geq0$.

\begin{Thm}
The conditional joint distribution of 
$m_0(\mathcal{Y}_k)$
and 
$Z(\mathcal{Y}_k)$
given the initial level and the initial phase is given as follows:
\begin{enumerate}[(a)]
\item If $0 \leq a \leq x$,
\begin{align}
\mathcal{P}^{\boldsymbol{\rho,q}}_+(a,x,y,k)
=\boldsymbol{q}_{+}(a,y,k),
\quad
\text{and}
\quad
\mathcal{P}^{\boldsymbol{\rho,q}}_-(a,x,y,k)
=\boldsymbol{q}_{-}(a,y,k).
\label{ZJointMinZTCaseA}
\end{align}
\item
If $0\leq x<a$,
\begin{align}
\mathcal{P}_{+}^{\boldsymbol{\rho,q}}(a,x,y,k)
&= 
\sum_{n=0}^{k-1}
{\boldsymbol{\Psi}^{(n)}}
\sum_{l=0}^{k-n-1}
\mathcal{P}_{-}^{\boldsymbol{\rho,q}}(a,x,y,k-n),
\label{ZJOINTPROBACASE2Plus}
\\
\mathcal{P}_{-}^{\boldsymbol{\rho,q}}(a,x,y,k)
&= 
\sum_{n=0}^{k-1}
\boldsymbol{W}^{(n)}_{a-x}
\boldsymbol{q}_{-}(x,y,k-n).
\label{ZJOINTPROBACASE2Minus}
\end{align}
\end{enumerate}
\end{Thm}

\begin{proof}
Equation (\ref{ZJointMinZTCaseA})
is easily justified:
$m_0({\mathcal{Y}_k})$
is at most equal to $a$ so that
 if
$x\geq a \geq 0$,
 \begin{equation*}
  \mathcal{P}_i^{\boldsymbol{\rho,q}}(a,x,y,k)
 =
  \mathbb{P}_{i}^{a}\left[Z(\mathcal{Y}_k)\leq y\right]
=  {q}_i(a,y,k),
 \end{equation*}
 by definition,
and is
 given in 
 (\ref{defdeqvect}).

 Take now
 $0\leq x<a$.
 Equation (\ref{ZJOINTPROBACASE2Minus}) is interpreted as follows.
 Given a phase in $\mathcal{S}_-$,
 the process must reach down to $x$ during the $n$-th exponential interval, for some $n$,
 and, 
 being at level $x$,
 it has to be at most equal to $y$ at the end of the remaining time.
 The probability of the first part of the trajectory is 
$\boldsymbol{W}^{(n)}_{a-x}$
and the probability of the second part is 
$\boldsymbol{q}_{-}(x,y,k-n)$.
 
Equation (\ref{ZJOINTPROBACASE2Plus})
is found by noting that the process reaches  level $x$, given the initial level $a$
in a phase of $\mathcal{S}_+$,
by first returning to level $a$ in one of the  $k$  exponential stages, with probability
${\boldsymbol{\Psi}^{(n)}}$,
$n\in\{0,...,k-1\}$. Then the situation is 
the same as in the previous case.
\end{proof}

\subsection{Joint distribution of $M_0(T)$ and $Z(T)$}

 Denote the joint distribution of $M_0(\mathcal{Y}_k)$ and  $Z(\mathcal{Y}_k)$ by
 $ \mathcal{P}^{\boldsymbol{\delta,q}}(a,x,y,k)$,
 which is the vector with components
 \begin{equation}
 \mathcal{P}_i^{\boldsymbol{\delta,q}}(a,x,y,k)
 =
 \mathbb{P}_{i}^{a}\left[M_0(\mathcal{Y}_k)\leq x,Z(\mathcal{Y}_k)\leq y\right],
 \end{equation}
 where  $i\in\mathcal{S}$,
 $1\leq k \leq L$,
 $a,x\geq0$.
To simplify the notation in what follows, we define the vector of probabilities $\boldsymbol{\mathcal{Q}}(a,x,y,k)$
 with components
\begin{align}
\mathcal{Q}_i(a,x,y,k)
=
 \mathbb{P}_{i}^{a}\left[M_0(\mathcal{Y}_k)> x,Z(\mathcal{Y}_k)\leq y
 \right].
 \label{ComplJointProbaMaxFF}
 \end{align}
 Of course, we have
\begin{align}\label{JointPM}
\mathcal{P}^{\boldsymbol{\delta,q}}_{}(a,x,y,k)
&=
\boldsymbol{q}(a,y,k)
-
\boldsymbol{\mathcal{Q}}(a,x,y,k),
\end{align}
where 
$\boldsymbol{q}(a,x,k)$
is given in the Theorem  \ref{DistribTimeTFF}.

\begin{Thm}
 The joint distribution of $M_0(\mathcal{Y}_k)$ and $Z(\mathcal{Y}_k)$ given the initial level and the initial phase is
 given by (\ref{JointPM})
 where $\boldsymbol{\mathcal{Q}}(a,x,y,k)$
 is characterized
  as follows. 
  
\begin{enumerate}[(a)]
\item Take $0=a\leq x $ and $1 \leq k \leq L$.
\begin{align}
\boldsymbol{\mathcal{Q}}_{+}(0,x,y,k)
&  = \left(I-\boldsymbol{\Psi}^{(0)}_x\Upsilon^{(0)}\right)^{-1}
\Big(
\sum_{n=0}^{k-1}
\boldsymbol{\Lambda}^{(n)}_{x}
\boldsymbol{q_{+}}(x,y,k-n)
\nonumber
\\
& \quad
+ \sum_{\overset{0\leq m,n}{1\leq m+n\leq k-1}}
\boldsymbol{\Psi}^{(n)}_{x}
\Upsilon^{(m)}
\boldsymbol{\mathcal{Q}}_{+}(0,x,y,k-m-n)
\Big),
\label{QPLUS0XYKeq}
\\
\boldsymbol{\mathcal{Q}}_{-}(0,x,y,k)
&=
\sum_{n=0}^{k-1}
\Upsilon^{(n)}
\boldsymbol{\mathcal{Q}}_{+}(0,x,y,k-n),
\label{QMoins0XYKeq}
\end{align}

\item Take $a>0$.

\begin{enumerate}[(i)]
\item
Take $x\geq a$.
\begin{align}
{\boldsymbol{\mathcal{Q}}_{+}(a,x,y,k)}
&  = \Big(I-
\boldsymbol{\Psi}^{(0)}_{x-a}
\hat{\boldsymbol{\Psi}}^{(0)}_{a}
\Big)^{-1}
\Bigg(
\sum_{n=0}^{k-1}
\boldsymbol{\Lambda}^{(n)}_{x-a}
\boldsymbol{q_{+}}(x,y,k-n)
\nonumber
\\
& \quad\quad\quad\quad \quad\quad\quad\quad\quad + 
\sum_{\overset{0\leq m,n}{1\leq m+n\leq k-1}}
\boldsymbol{\Psi}^{(m)}_{x-a}
\mathcal{L}(n,k-m)
\Bigg),
\label{QPlusAXYeq}
\\
\boldsymbol{\mathcal{Q}}_{-}(a,x,y,k)
&=
\sum_{n=0}^{k-1}
\mathcal{L}(n,k),
\label{QMoinsAXYKeq}
\end{align}
where
\begin{equation}\label{EqofL}
\mathcal{L}(n,k)=
\hat{\boldsymbol{\Psi}}^{(n)}_{a}
\boldsymbol{\mathcal{Q}}_{+}(a,x,y,k-n)
+
 \hat{\boldsymbol{\Lambda}}^{(n)}_{a}
\boldsymbol{\mathcal{Q}}_{-}(0,x,y,k-n).
\end{equation}

\item Take $x<a$.
  \begin{align}
\boldsymbol{\mathcal{Q}}_{+}(a,x,y,k)
=
\boldsymbol{q}_{+}(a,y,k)
\qquad
\text{and}
\qquad
\boldsymbol{\mathcal{Q}}_{-}(a,x,y,k)
=
\boldsymbol{q}_{-}(a,y,k).
\end{align}
\end{enumerate}
\end{enumerate}
\end{Thm}

\begin{proof}
Take $a=0$
and
 $x\geq 0$.
Then,
$\boldsymbol{\mathcal{Q}}_{+}(0,x,y,k)
$
is as follows
\begin{align}
\boldsymbol{\mathcal{Q}}_{+}(0,x,y,k)
 &=
 \sum_{n=0}^{k-1}
\boldsymbol{\Lambda}^{(n)}_{x}
\boldsymbol{q_{+}}(x,y,k-n)
 \nonumber
\\
&\quad +
\sum_{\overset{0\leq m,n}{1\leq m+n\leq k-1}}
\boldsymbol{\Psi}^{(n)}_{x}
\Upsilon^{(m)}
\boldsymbol{\mathcal{Q}}_{+}(0,x,y,k-m-n),
\label{QPLUS0XYKeqProof}
\end{align}
for
$y\geq0$.
Indeed, the first term, is the probability that the level reaches level $x$ 
 in stage $n$
before returning to level $0$; then one have the conditional probability that given that the process is in level $x$ in a phase of 
$\mathcal{S}_+$,
the process has to be less than $y$ at time
$\mathcal{Y}_{k-n}$.
The second term is the probability that the level returns to $0$ in a phase of $\mathcal{S}_-$ before level $x$ is reached, with probability $\boldsymbol{\Psi}^{(n)}_{x}$, then it spends some time in level $0$, with probability $\Upsilon^{(m)}$,
and when the phase changes to $\mathcal{S}_+$,
the situation is as the initial one, except that there are $k-n-m$ stages left. After a reorganization of the terms of
(\ref{QPLUS0XYKeqProof})
we  obtain
(\ref{QPLUS0XYKeq}).
To prove (\ref{QMoins0XYKeq}) is obvious.

Take $x\geq a >0$.
If the initial phase is in $\mathcal{S}_-$,
there are two ways to reach $x$ before $\mathcal{Y}_k$ and to be below $y$ at time $\mathcal{Y}_k$, which are
 combined in $\mathcal{L}(\cdot,\cdot)$ given in (\ref{EqofL}). The first is to return from below to level $a$ in a phase of $\mathcal{S}_+$  before time $\mathcal{Y}_k$ and before reaching level $0$, with probability 
$\hat{\boldsymbol{\Psi}}^{(n)}_{a}$;
 then the situation is as if the process starts in 
 $\mathcal{S}_+$
 but with $k-n$
 stages left only, thus we have to multiply this by 
 $\boldsymbol{\mathcal{Q}}_{+}(a,x,y,k-n)$.
 The second way is to reach level $0$
 before returning to $a$,
 with probability $ \hat{\boldsymbol{\Lambda}}^{(n)}_{a}$
 and then again, the situation is as if the process starts in a phase of $\mathcal{S}_-$, in level $0$, this probability is given by 
 $\boldsymbol{\mathcal{Q}}_{-}(0,x,y,k-n)$.

The proof  of (\ref{QPlusAXYeq})
is nearly identical to the proof of (\ref{QPLUS0XYKeq}).
The first term in the big bracket is 
equal to the first term in 
(\ref{QPLUS0XYKeqProof})
starting from $a$ instead of $0$;
but the second one is slighlty different because when the level first come back to the initial level $a$, before crossing level $x$, with probability
$\boldsymbol{\Psi}^{(m)}_{x-a}$, 
it does so in a phase of $\mathcal{S}_-$.
 Then there are two possibilities
which are given in $\mathcal{L}(n,k-m)$ as there are $k-m$
stages left.
A simple algebraic manipulation leads then to equation
(\ref{QPlusAXYeq}).

The case $0\leq x<a$ is obvious.
\end{proof}

Remark that here again,
we have important simplifications when $k=1$
for equations
(\ref{QPLUS0XYKeq})
and
(\ref{QPlusAXYeq})
 as they contain sums on empty sets.

\section{Numerical illustration}\label{ILL}
We illustrate our results with the following example: the value of some asset normally evolves in one of two environments: it increases in environment 1 and decreases otherwise. Occasionally, the rates of variation become much higher, for short periods of time, indicating unusually high activity.
Precisely, the generator is
\begin{equation*}
A=
\begin{array}{cc}
 & 1\hspace{1.5cm}2\hspace{1.5cm}3\hspace{1.5cm}4\\
\begin{array}{c}
1\\
2\\
3\\
4
\end{array} & \left[\begin{array}{cc|cc}
-\lambda-\omega & \lambda & p\omega & (1-p)\omega\\
\lambda & -\lambda-\omega & p\omega & (1-p)\omega\\
\hline \mu & 0 & -\mu-\beta & \beta\\
0 & \mu & \beta & -\mu-\beta
\end{array}\right]
\end{array}.
\end{equation*}
Phases $1$ and $2$  correspond to the calm environment, (CE) with $c_1=2$
and $c_2=-1$,
to reflect the fact that the value of the asset is generally increasing.
Phases $3$ and $4$
correspond to the excited environment (EE),
with $c_3=10$,
$c_4=-10$.
The unit of time is one week,
and we take $\lambda = 1$.
The parameter 
$\omega$
is equal to $0.25$
and $\mu=1$
so that the process remains for four weeks in the average in the CE before moving to the EE where
it remains for one week, on average.
Finally, $\beta=7$,
so that during the EE,
switching from increase to decrease occurs every day.

We take  $\varphi(0)=2$
 so that the level starts decreasing
at the rate $-1$ and we assume that
$X(0)=0$.

The distribution  for $L$ fixed and different maturities $T$
is depicted on Figure \ref{Fig2EX}.
Globally we see that $X(T)$ increases when $T$ increases. This is du to the fact that $c_1=2$
and $c_2=-1$
so that the stationary  drift 
	 $\boldsymbol{\alpha} C \boldsymbol{1}$
is slightly positive,
where
$\boldsymbol{\alpha}$ is the stationary probability vector of $A$
and $C$ is the rate matrix.

\begin{figure}
\begin{center}
\includegraphics[scale=0.5]{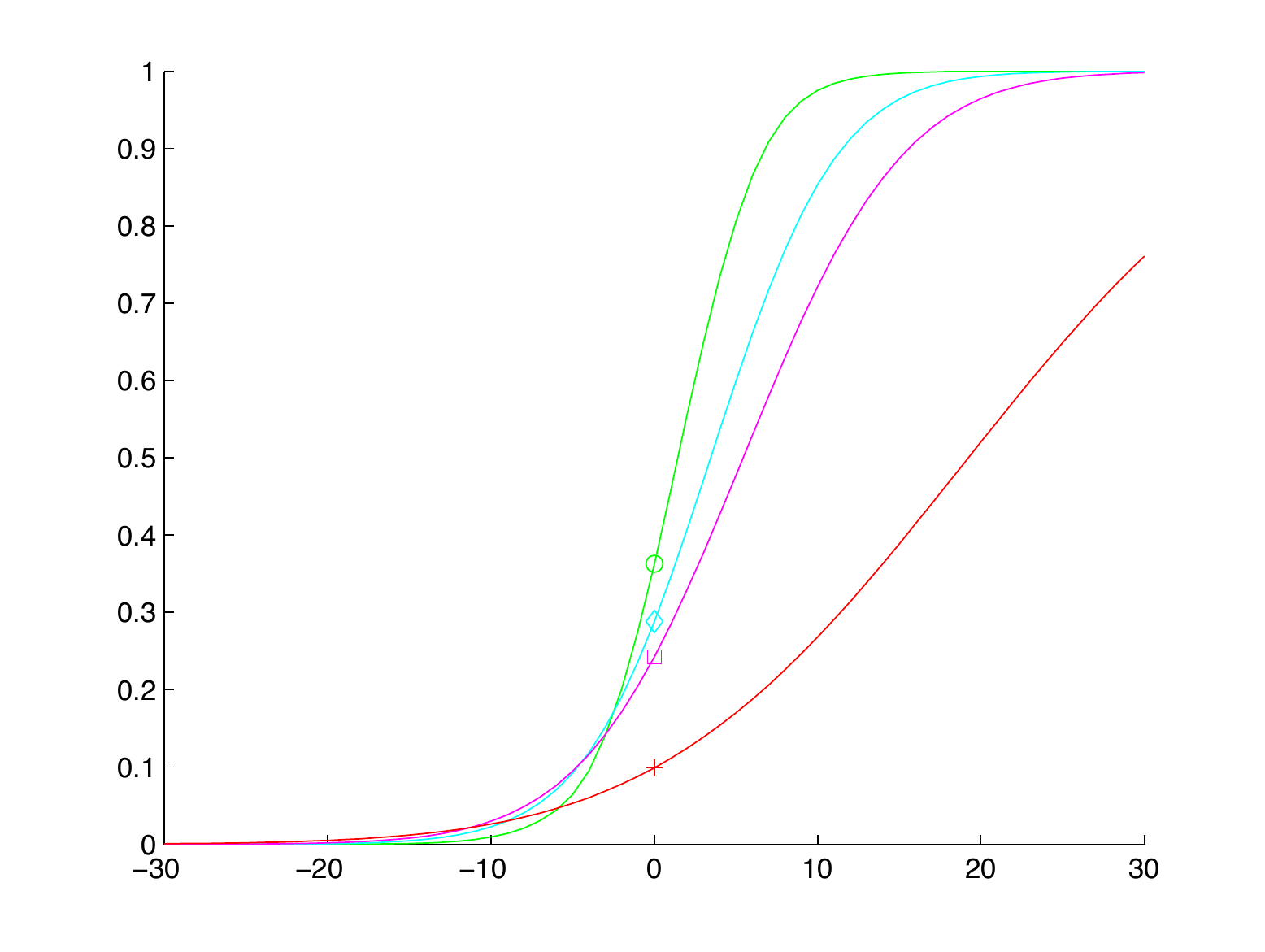}
\caption{Distribution function of the level at different maturities, where $L=30$ is constant.
 The curve with the symbol
  {\textcolor{green}{$\ocircle$}}  for $T=5$,
{\textcolor{cyan}{$\Diamond$}}  for $T=10,$ 
{\textcolor{magenta}{$\square$}}  for $T=15$
and
{\textcolor{red}{$+$}}  for $T=50$.
 }\label{Fig2EX}
 \end{center}
\end{figure}

We plot on   Figure \ref{Figure1Ex} the distribution function of the level at Erlang maturities
$T\sim\mbox{Erl}\left(L/\theta,L\right)$ for $\theta=10$  and $L=1,2,5,10$ and $30$.
\begin{figure}
\includegraphics[scale=0.5]{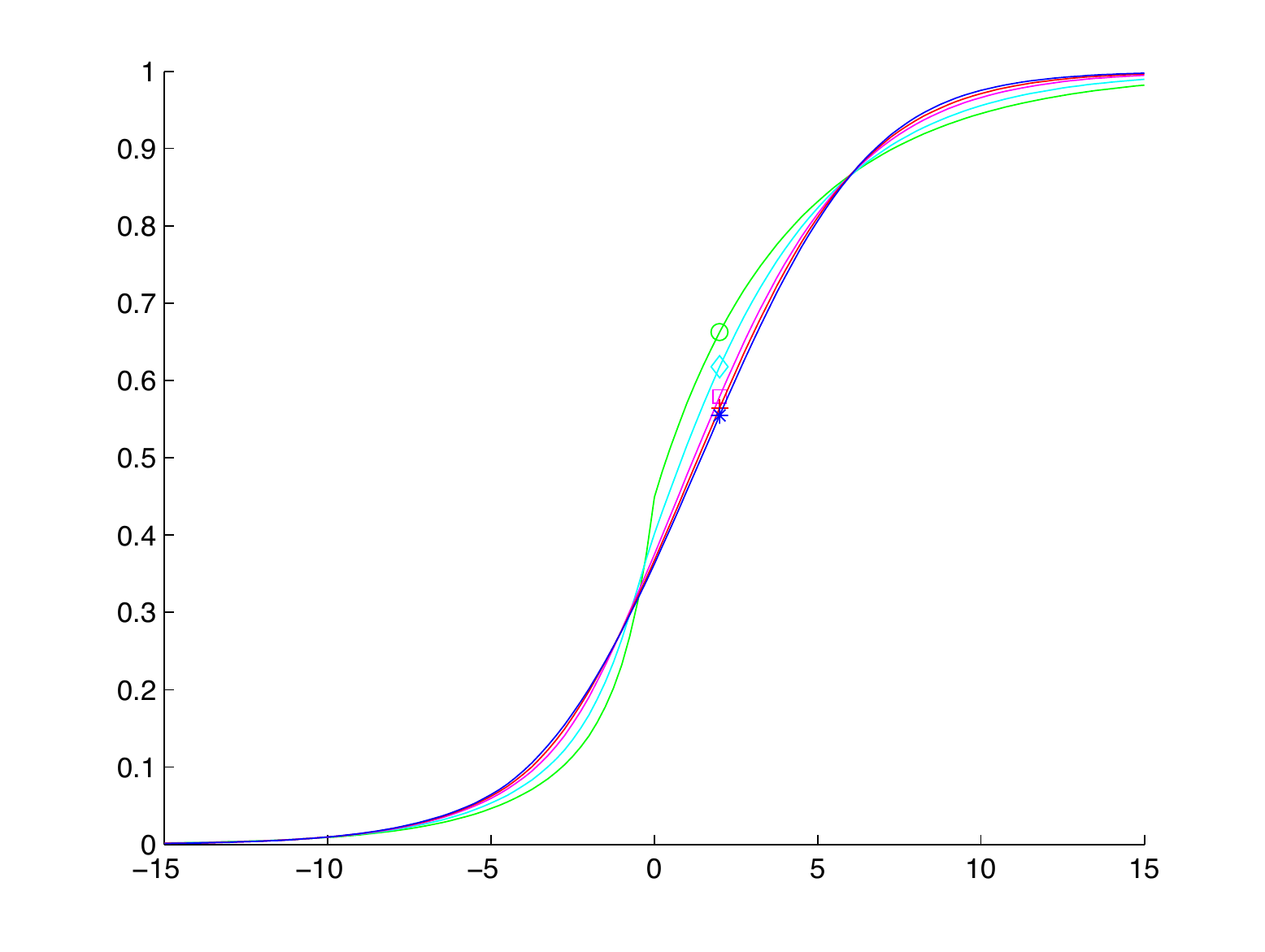} 
\includegraphics[scale=0.5]{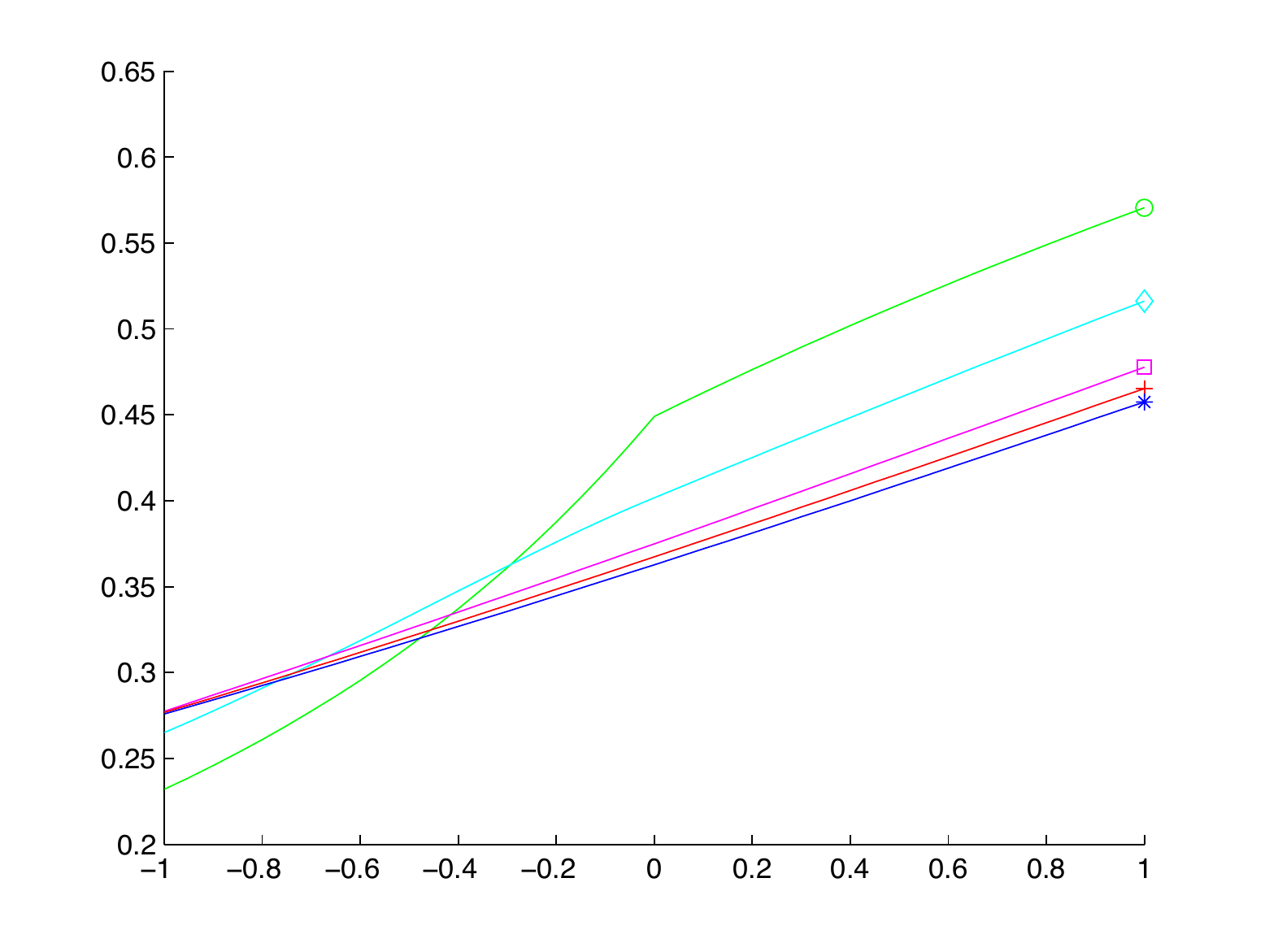}
\caption{Distribution function of the level at maturity, where the Erlang approximating
time has mean $\theta=10$,
conditionally given $X(0)=0$
and $\varphi(0)=2$.
The curve with the symbol  {\textcolor{green}{$\ocircle$}} is
for the parameter $L=1$,
{\textcolor{cyan}{$\Diamond$}} for $L=2,$ 
{\textcolor{magenta}{$\square$}} for $L=5,$
{\textcolor{red}{$+$}} for $L=10$ and 
{\textcolor{blue}{$*$}} for $ $$L=30.$ }\label{Figure1Ex}
\end{figure}
In general, the curves are quite smooth, with the exception of the density for $L=1$
if $\phi(0)$ is in $\mathcal{S}_{-}$:
we
display in greater detail the functions in a small interval around $0$ 
and we clearly see that the curve for $L=1$
has a different aspect at $x=0$.
We write 
\begin{equation*}
\boldsymbol{r}'_{-}(0^{-},1)
=
\lim_{x\uparrow0}
\frac{\partial}{\partial x}\boldsymbol{r}_{-}(x,1)
 \text{\quad and \quad }
 \boldsymbol{r}'_{-}(0^{+},1)
=
\lim_{x\downarrow0}
\frac{\partial}{\partial x}\boldsymbol{r}_{-}(x,1).
\end{equation*}
It appears from Figure \ref{Figure1Ex}
that the density  is discontinuous
there with 
$\boldsymbol{r}'_{-}(0^{-},1)>
 \boldsymbol{r}'_{-}(0^{+},1)$.
 This we  interpret as follows.
 In general, the value for $X(T)$ is the result of multiple
 changes, which results in smooth curves.
 If $L=1$
 and $\phi(0)=i \in \mathcal{S}_{-}$,
 however there is a small probability, 
 of the order of $\nu h$
 that the process is absorbed almost at once,
 in a small interval of length $h$.
If that happens, the fluid level
 is $c_i h<0$ and close to $0$.
 This explains why the difference between the limits of the density from below and from above is $\nu / \vert c_i \vert $,
 as we now show
 precisely.
On the one hand,   
\begin{eqnarray*}
\boldsymbol{r}'_{-}(0^{-},1)
&=&
\lim_{x\uparrow0}
 \frac{\partial}{\partial x}
 \boldsymbol{W}_{\left|x\right|}
 \hat{\boldsymbol{h}}(1),\\
 & = & -\boldsymbol{U}\hat{\boldsymbol{h}}(1),\\
 & = & -\left|C_{-}^{-1}\right|(A_{--}-\nu I)\hat{\boldsymbol{h}}(1)-\left|C_{-}^{-1}\right|A_{-+}\left(\boldsymbol{1}-\boldsymbol{h}(1)\right),
\end{eqnarray*}
by Proposition \ref{Linkhhhat},
and
on the other hand we have  that
\begin{eqnarray*}
\boldsymbol{r}'_{-}(0^{+},1)
&=&
\lim_{x\downarrow0}
 \frac{\partial}{\partial x}(\boldsymbol{1}-\hat{\boldsymbol{\Psi}}\hat{\boldsymbol{W}}_{\left|x\right|}\boldsymbol{h}(1)),\\
 & = & -\boldsymbol{\hat{\Psi}\hat{U}h}(1),\\
 & = & (-\boldsymbol{\hat{\Psi}}C_{+}^{-1}(A_{++}-\nu I)-\boldsymbol{\hat{\Psi}}C_{+}^{-1}A_{+-}\boldsymbol{\hat{\Psi}})\boldsymbol{h}(1),\\
 & = & 
 \left(
 \left|C_{-}^{-1}\right|(A_{--}-\nu I)\boldsymbol{\hat{\Psi}}+\left|C_{-}^{-1}\right|A_{-+}
 \right)
 \boldsymbol{h}(1),
 \end{eqnarray*}
 by Theorem \ref{RiccatiPsiEq}.
By Proposition \ref{Linkhhhat} we obtain
 \begin{eqnarray*}
 \boldsymbol{r}'_{-}(0^{+},1)
  & = & \left|C_{-}^{-1}\right|(A_{--}-\nu I)(\boldsymbol{1}-\hat{\boldsymbol{h}}(1))+\left|C_{-}^{-1}\right|A_{-+}\boldsymbol{h}(1),
\end{eqnarray*}
so that
\begin{eqnarray*}
\boldsymbol{r}'_{-}(0^{-},1)
-\boldsymbol{r}'_{-}(0^{+},1)
& = & -\left|C_{-}^{-1}\right|(A_{--}-\nu I)\boldsymbol{1}-\left|C_{-}^{-1}\right|A_{-+}\boldsymbol{1}\\
  & = & \nu\left|C_{-}^{-1}\right|\boldsymbol{1},
\end{eqnarray*}
 as $Q\boldsymbol{1}=0$,
 or 
 \begin{equation*}
r'_{i}(0^{-},1)
=r'_{i}(0^{+},1)+\frac{\nu}{\left|c_{i}\right|},
\end{equation*}
for any $i\in\mathcal{S}_{-}$.

\begin{Rem} 
{\bf Computational issues}

{\rm{
In this example, the total number $Lm$ of phases is sufficiently small and we have computed the matrix 
$\boldsymbol{W}_{x}=\exp(\boldsymbol{U}x)$
 by using the function
\texttt{expm}
 from MATLAB
 \cite{MATLAB}.
 When $Lm$ is large, we need to reduce the cost of computing $\boldsymbol{W}_{x}$
 which is $O(L^{3}m^{3})$
 (see Moler and Van Loan \cite{mv03})
 unless one uses structural properties of $U$.
 
 One approach is to apply
Algorithm 1 of Xue and Ye
\cite{xue2013computing}
which is designed for the computation of general 
exponentials of essentially non-negative matrices entry-wise to high relative accuracy.

In addition to computing $\boldsymbol{W}_{x}$
fast and accurately,
we wish to compute the blocks 
$ \boldsymbol{W}^{(0)}_{x},\cdots, \boldsymbol{W}^{(n)}_{x}$  only,
as they completely  specify $\boldsymbol{W}_{x}$,
 and this allows to fully benefot from the decomposition approach followed in this paper.
 Efficient algorithms to that effect are developed in
Bini {\it{et al.}} \cite{bini2014toepliz}.
One method considered there 
consists in exploiting the Toeplitz structure to specialize the shifting and Taylor series method of  
\cite{xue2013computing}.

Another highly efficient method, to reduce the computation cost is based on 
the block-circulant matrix method:
we define the matrix 
\begin{equation*}
\boldsymbol{V}_{\epsilon}
=\left[\begin{array}{cccccc}
\boldsymbol{U}^{(0)} & \boldsymbol{U}{}^{(1)} & \boldsymbol{U}^{(2)} &  &  & \boldsymbol{U}{}^{(L-1)}\\
\epsilon\boldsymbol{U}^{(L-1)} & \boldsymbol{U}^{(0)} & \boldsymbol{U}{}^{(1)} &  &  & \boldsymbol{U}{}^{(L-2)}\\
\epsilon\boldsymbol{U}{}^{(L-2)} &\epsilon \boldsymbol{U}^{(L-1)} & \boldsymbol{U}{}^{(0)} &  &  & \boldsymbol{U}^{(L-3)}\\
 \vdots & \vdots & \ddots & \ddots &  & \vdots\\
 &  &  & \\
 & & & & & \\
\epsilon\boldsymbol{U}{}^{(1)} & \cdots &  &  &  \epsilon\boldsymbol{U}{}^{(L-1)}& \boldsymbol{U}\mathcal{}^{(0)}
\end{array}\right].
\end{equation*}
where $\epsilon$
is some small number.
This is a block $\epsilon-$circulant matrix which may be block-diagonalized by Fast Fourier Transforms techniques,
so that the computation of 
$\exp(\boldsymbol{V}_{\epsilon}x)$
is reduced to 
$O(m^{2} L \log_{2}L)$
the computation of $L$ exponentials of matrices of order $m$,
and serves as a close approximation of 
$\boldsymbol{W}_{x}$.

A third approach investigated in
Bini {\it{et al.}} \cite{bini2014toepliz}
is as follows: one defines the matrices 
$S^{(0)},...,S^{(K-1)}$,
for $K\geq L$,
with 
\begin{numcases}
{ S^{(i)}=}
\boldsymbol{U}^{(i)},
&
\mbox{for  } $i \leq L-1$, \nonumber
\\
0, &
\mbox{for  } $L\leq i \leq K-1$,\nonumber

\end{numcases}
and  the block-circulant matrix
\begin{equation*}
S
=\left[\begin{array}{cccccc}
S^{(0)} & S^{(1)} & S^{(2)} & \cdots &  & S^{(K-1)}\\
S^{(K-1)} & S^{(0)} & S^{(1)} & \cdots &  & S^{(K-2)} \\
S^{(K-2)} & S^{(K-1)} & S^{(0)} & \cdots &  & S^{(K-3)}\\
\vdots &  & \ddots & \ddots  &  & \vdots \\
& & & & & \\
S^{(1)} & \cdots & & & S^{(K-1)}& S^{(0)}  
\end{array}\right].
\end{equation*}
The matrix $\exp(Sx)$
may be efficiently computed with a complexity 
$O(m^{2} K \log_{2}K)$
plus the cost of $K$ exponentials of matrices of order $m$,
 by means of FFT techniques and,
 for $K$ large enough,
 the blocks 
 $[\exp(Sx)]_{0,i}$,
 for $0\leq i \leq L-1$,
 constitute a good approximation of 
 $ \boldsymbol{W}^{(0)}_{x},\cdots, \boldsymbol{W}^{(L-1)}_{x}$.
 
 Full details about these three approximation methods are given in Bini {\it{et al.}} \cite{bini2014toepliz}.
}}
\end{Rem}

\subsection*{Acknowledgment}
The authors thank the Minist\`{e}re de la Communaut\'{e} fran\c caise de Belgique for funding this research through the ARC grant
AUWB-08/13-ULB 5.

\bibliographystyle{plain}

\end{document}